\documentclass[11pt]{amsart}

\usepackage{amsmath, amsthm, amssymb}

\usepackage{diagrams}			
\usepackage{tensor}				%
\usepackage{GrCalc3}

\usepackage{geometry}
\usepackage{lscape}

\theoremstyle{plain}
\newtheorem{theorem}{Theorem}[section]
\newtheorem{corollary}[theorem]{Corollary}
\newtheorem{lemma}[theorem]{Lemma}
\newtheorem{proposition}[theorem]{Proposition}

\theoremstyle{definition}
\newtheorem{definition}[theorem]{Definition}

\newtheorem{remark}[theorem]{Remark}

\numberwithin{equation}{section}
\def\BM{\mathrm{BM}}
\def\Br{\mathrm{Br}}
\def\Pic{\mathrm{Pic}}
\def\Aut{\mathrm{Aut}}
\newcommand{\hide}[1]{}

\newcommand{\dome}[1]{\hide{#1}}

\allowdisplaybreaks[1]

\newcommand{\inv}{^{-1}} 
\newcommand{\C}{\mathcal{C}}
\newcommand{\U}{\mathcal{U}}

\newcommand{\BC}{{}^B\C}
\newcommand{\LC}{{}^L\C}
\newcommand{\FC}{{}^F\C}
\newcommand{\D}{\mathcal{D}}
\newcommand{\E}{\mathcal{E}}
\newcommand{\mO}{\mathcal{O}}
\newcommand{\R}{\mathcal{R}}
\newcommand{\yd}{{}_H^H\mathcal{YD}}

\newcommand{\dy}{\mathcal{YD}_{H^*}^{H^*}}
\newcommand{\mr}{{}_H^R\mathcal{M}}
\newcommand{\mri}{\tensor*[^{R\inv}_H]{\mathcal{M}}{}} 

\newcommand{\ot}{\otimes}

\newcommand{\ra}{\rightarrow}
\newcommand{\lra}{\longrightarrow}

\newcommand{\va}{\varepsilon}
\newcommand{\st}{\ | \ }

\newcommand{\ab}{\alpha(B)}
\newcommand{\am}{\alpha(M)}
\newcommand{\om}{\omega(M)}
\newcommand{\aam}{\alpha_A(M)}
\newcommand{\oam}{\omega_A(M)}
\newcommand{\bL}{\beta(L)}

\newcommand{\btc}{\boxtimes_{\C}}

\newcommand{\oA}{\overline{A}}
\newcommand{\oB}{\overline{B}}
\newcommand{\ol}[1]{\overline{#1}}
\newcommand{\oa}{\overline{a}}
\newcommand{\ob}{\overline{b}}

\newcommand{\hm}[1]{_{(#1)}}

\newcommand{\hrm}[1]{_{( #1)}}

\newcommand{\blm}[1]{^{[#1]}}

\newcommand{\Blm}[1]{#1^{[-1]}\otimes #1^{[0]}}

\newcommand{\brm}[1]{^{[ #1 ] }}
\newcommand{\Brm}[1]{#1^{[ 0 ] } \otimes #1^{[ 1 ]}}

\newcommand{\by}[1]{\tag*{by \eqref{#1}}}

\newcommand{\canbr}{can_+}

\newcommand{\gammar}{\gamma}

\newcommand{\morph}{\xi}

\newcommand{\bg}[1]{\mathrm{BiGal}(#1)}

\newcommand{\gqc}[1]{\mathrm{Gal}^{qc}(#1)}

\newcommand{\autbr}{\tensor*[]{\mathrm{Aut}}{^{br}_{}}(\yd,\mr)} 
\newcommand{\autbra}{\tensor*[]{\mathrm{Aut}}{^{br}_{\scriptscriptstyle {(\mathbb{A})}}}(\yd,\mr)} 
\newcommand{\autaf}[2]{\tensor*[]{\mathrm{Aut}}{_{\scriptscriptstyle {(\mathbb{A})}}}(#1,#2)}

\newcommand{\ct}{\square}

\newcommand{\gnotc}[1]{\gnot{\hspace{4mm}#1}}

\newcommand{\ppi}{\tilde{\pi}}

\newcommand{\M}{\mathcal{M}}
\newcommand{\N}{\mathcal{N}}

\newcommand{\Ae}{A^e}
\newcommand{\op}{^{op}}

\newcommand{\HM}{{}_H\M}

\newcommand{\MHs}{\M^{H^*}}

\newcommand{\lhu}{\leftharpoonup}

\newcommand{\tr}{{}_RH}
\newcommand{\ud}{\underline{\Delta}}
\newcommand{\uS}{\underline{S}}
\newcommand{\uR}{\underline{R}}

\newcommand{\galqc}{\mathrm{Gal}^{qc}(\tr)}
\newcommand{\trbicom}{\tensor[^{\tr}]{\overline{(\HM)}}{^{\tr}}}
\newcommand{\trcom}{\tensor[^{\tr}]{(\HM)}{}}
\newcommand{\ff}{F}

\begin{document}
\title[Tensor autoequivalences and equivariant Brauer groups] {Braided autoequivalences and the equivariant Brauer group of a quasitriangular Hopf algebra}
\author{Jeroen Dello}
\address{Department of Mathematics, University of Hasselt,\\
 Agoralaan 1, 3590 Diepenbeek, Belgium}
\email{dellojeroen@gmail.com}
\author{Yinhuo Zhang}
\address{Department of Mathematics, University of Hasselt,\\
 Agoralaan 1, 3590 Diepenbeek, Belgium}
\email{yinhuo.zhang@uhasselt.be}
\date{}

\maketitle

\begin{abstract}
Let $(H, R)$ be a finite dimensional quasitriangular Hopf algebra over a field $k$, and $_H\mathcal{M}$ the representation category of $H$. In this paper, we study the braided autoequivalences of the Drinfeld center $^H_H\mathcal{YD}$ trivializable on $_H\mathcal{M}$. We establish a group isomorphism between the group of those autoequivalences and the group of quantum commutative bi-Galois objects of the transmutation braided Hopf algebra $\tr$. We then apply this isomorphism to obtain a categorical interpretation of the exact sequence of the equivariant Brauer group $\mathrm{BM}(k, H,R)$ in \cite{zhang}.  To this aim, we have to develop the braided bi-Galois theory initiated by Schauenburg in \cite{schauenburgbr1,schauenburgbr2}, which generalizes the Hopf bi-Galois theory over usual Hopf algebras to the one over braided Hopf algebras in a braided monoidal category.
\end{abstract}

\section*{Introduction}

Let $(H, R)$ be a finite dimensional quasitriangular Hopf algebra over a field $k$, and $\C$ the representation category of $H$. In \cite{zhang}, the second author studied the equivariant Brauer group $\mathrm{BM}(k,H,R)$ of $H$-Azumaya algebras and established an exact sequence:
\begin{equation}\label{exactss}
1 \lra \Br(k) \lra \BM(k,H,R) \overset{\ppi}\lra \gqc{\tr}
\end{equation}
where  $\gqc{\tr}$ is the group of quantum commutative $\tr$-bi-Galois objects. In \cite{ZZ}, it was showed that the group $\gqc{\tr}$ is embedded into the  group $\mathrm{Aut}^{br}(\mathcal{Z(C)}:\C)$ of the braided autoequivalences of the Drinfeld center $\mathcal{Z(C)}$ trivializable on $\C$. The embedding is based on the crucial observation \cite[Theorem 2.5]{ZZ} that the Drinfeld center $\mathcal{Z(C)}$, isomorphic to the Yetter-Drinfeld $H$-module category $^H_H\mathcal{YD}$, is isomorphic to the comodule category $^{\tr}\C$, where $\tr$ is the transmutation braided Hopf algebra in $\C$. That is, every bi-Galois object over $\tr$ defines a $\C$-linear autoequivalence of $^{\tr}\C$ which preserves the braiding of $^{\tr}\C$, and hence giving a braided autoequivalence of the Drinfeld center $\mathcal{Z(C)}$ trivializable on $\C$. 

Now a natural question arises: is every such a braided autoequivalence given by a quantum commutative bi-Galois object over $\tr$. The answer is affirmative. To prove this, we have to develop the braided bi-Galois theory initiated by Schauenburg in \cite{schauenburgbr1,schauenburgbr2}. 

Assume $\C$ is a braided monoidal category, $L$ and $B$ are two (flat) Hopf algebras in $\C$ and $A$ is a (faithfully flat) $L$-$B$-bi-Galois object in $\C$. Schauenburg showed in \cite{schauenburgbr2} that the cotensor functor $A \ct_B - : {}^B\C \ra {}^L\C$ defines a tensor equivalence of categories. To show the converse, we observe that the categories ${}^L\C$ and ${}^B\C$ are naturally right $\C$-module categories and the cotensor functor $A \ct_B - $ is a right $\C$-module functor (or shortly said, right $\C$-linear). Moreover, 
the monoidal functor $A \ct_B -$ satisfies some diagram which more or less reflects the compatibility between the monoidal structure of the functor and the braiding when one operand in ${}^B\C$ has a trivial $B$-comodule structure, see Lemma \ref{lemmaalphaaa}.
We show that these two conditions are also sufficient  for a tensor equivalence of the form ${}^B\C \ra {}^L\C$ to induce a (faithfully flat) bi-Galois object. To be more precise, if $\alpha : \BC \rightarrow \LC$ is a tensor equivalence functor which is right $\C$-linear (or trivializable on $\C$) and satisfies the compatibility diagram $(\mathbb{A})$ in Lemma \ref{lemmaalphaaa}, then $\ab$ is a faithfully flat $L$-$B$-bi-Galois object in $\C$.
As a consequence, we obtain a group isomorphism between the group of faithfully flat $B$-bi-Galois BiGal$(B)$ and the group of tensor autoequivalences of $\BC$ which are right $\C$-linear and satisfy the aforementioned diagram. This isomorphism restricts to an isomorphism between the subgroup Aut$^{bc}(\BC)$ of braided (tensor) autoequivalences of $\BC$ and the subgroup $\gqc{\tr}$ of quantum commutative $B$-bi-Galois objects thanks to \cite[Corollary 3.12]{ZZ}. Therefore, we obtain a categorical interpretation of the group $\gqc{\tr}$ in the exact sequence (\ref{exactss}).

Subsequently, we would like to characterise the the Brauer group $\BM(k,H,R)$ and the group morphism $\ppi: \BM(k,H,R) \ra \gqc{\tr}$ in the sequence (\ref{exactss}) in terms of braided autoequivalences of $\yd$ trivializable on $\mr$. Inspired by the work of Davydov and Nikshych in \cite{davydovnikshych} where the authors study the Brauer-Picard group  BrPic$(\C)$ of a finite tensor category $\C$, which consists of equivalence classes of invertible $\C$-bimodule categories (see \cite{eno}),  we consider the Brauer-Picard group of the braided category $^R_H\mathcal{M}$. In \cite{davydovnikshych}, the authors
showed that there exists a group isomorphism from BrPic$(\C)$ to  Aut$^{br}(\mathcal{Z}(\C))$, the group of braided autoequivalences of $\mathcal{Z}(\C)$. Note that the isomorphism holds when $\C$ is a linear fusion category over an algebraically closed field. In general, we are not sure whether it's true. In case $k$ is not algebraically closed, the two groups are not nicessary isomorphic.  If $\C$ is also braided, one can consider the subgroup Pic$(\C)$ of BrPic$(\C)$ consisting of isomorphism classes of one-sided \mbox{invertible} $\C$-module categories. The isomorphism from BrPic$(\C)$ to Aut$^{br}(\mathcal{Z}(\C))$ restricts to an isomorphism from Pic$(\C)$ to Aut$^{br}(\mathcal{Z}(\C),\C)$. 

Now suppose $A$ is an Azumaya algebra in $\C$. The category $\C_A$ has a natural structure of a left $\C$-module category. As observed in \cite{davydovnikshych}, the Picard group of $\C$ is isomorphic to the group of Morita equivalence classes of exact Azumaya algebras. In case $\C=\HM$ and $(H,R)$ is a finite dimensional quasitriangular Hopf algebra, 
we show that any Azumaya algebra in $\HM$ is exact. Thus, the Picard group of $\HM$ is isomorphic to the Brauer group $\BM(k,H,R)$. In other words, we obtain a categorical characterization for the group $\BM(k,H,R)$. As a result of the foregoing identification, we are able to give a categorical characterization for the morphism \mbox{$\ppi : \BM(k,H,R) \ra \gqc{\tr}$} as well. This will be the second main result of this paper, that is, we have  a commutative diagram of group morphisms (see Theorem \ref{bmautbrasequence}):
\begin{diagram}[size=2em]
1	& \rTo	&	\Br(k)	&\rTo	& \BM(k,H,R)			&		\rTo^{\phantom{aa} \ppi}		&		\gqc{\tr}	\\
	&		&	\dTo^{\sim}&		&	\dTo^{\sim}		&										&		\dTo_{\sim}		\\
1	& \rTo	&	\Pic({}_k\M)	&\rTo	&	\Pic(\HM)		&		\rTo^{\tau}							&		\autbr
\end{diagram}

As mentioned before, the idea behind relating the morphism $\ppi$ to a morphism Pic$(\HM) \ra \autbra$ is partly inspired by \cite{davydovnikshych}, but our approach to the construction of the morphism comes from \cite[Lemma 4.5]{zhang}, and they may not be the same. If, however, the two constructions would coincide, we would obtain that the morphism $\BM(k,H,R) \ra \gqc{\tr}$ is surjective (under suitable conditions). At the moment, this is still an open problem. Note that the morphism $\BM(k,H,R) \ra \gqc{\tr}$ has a kernel while Davydov and Nikshych obtain an isomorphism Pic$(\C) \ra \Aut^{br}(\mathcal{Z}(\C),\C)$ in \cite{davydovnikshych}. However, they work over an algebraically closed field (of characteristic 0), in which case the classical Brauer group $\Br(k)$ is trivial and hence the morphism $\BM(k,H,R) \ra \gqc{\tr}$ is already injective. \hide{Relation with the morphism $\morph$}

\section{Preliminaries}\label{secprelim}

\subsection{Braided bi-Galois objects}
We will assume the reader is familiar with the theory of braided monoidal categories and with the notions of algebras, bialgebras, Hopf algebras, (co)modules and (co)module algebras in braided monoidal categories. For the details, the reader is referred to \cite{Maj}. Moreover, by Mac Lane's coherence theorem, we may and will assume the monoidal categories we work with are strict, i.e., the associativity and unity constraints are the identity.
\\ 
In the following, $(\C,\phi,I)$ denotes a strict braided monoidal category with (co-)equalizers. We will make use of graphical calculus using the following notation
$$
\phi_{M,N} \,
=
\gbeg{2}{3}
\got{1}{M}\got{1}{N}\gnl
\gbr\gnl
\gob{1}{N}\gob{1}{M}
\gend
\quad
\text{and}
\quad
\phi\inv_{N,M} \,
=
\gbeg{2}{3}
\got{1}{N}\got{1}{M}\gnl
\gibr\gnl
\gob{1}{M}\gob{1}{N}
\gend
$$
for the braiding and its inverse. Let $A$ be an algebra and $C$ a coalgebra in $\C$, we denote the multiplication and unit of $A$, and the comultiplication and counit of $C$ by
$$
\nabla_A \,
=
\gbeg{2}{3}
\got{1}{A}\got{1}{A}\gnl
\gmu\gnl
\gob{2}{A}
\gend
\quad , \quad
\eta_A \,
=
\gbeg{1}{3}
\gvac{1}\gnl
\gu{1}\gnl
\gob{1}{A}
\gend
\quad
\text{and}
\quad
\Delta_C \,
=
\gbeg{2}{3}
\got{2}{C}\gnl
\gcmu\gnl
\gob{1}{C}\gob{1}{C}
\gend 
\quad , \quad
\epsilon_C \,
=
\gbeg{1}{3}
\got{1}{C}\gnl
\gcu{1}\gnl
\gnl
\gend
$$
For a left $A$-module $M \in {}_A\C$ respectively a right $A$-module $M \in \C_A$, we denote
$$
\mu^- \,
=
\gbeg{2}{3}
\got{1}{A}\got{1}{M}\gnl
\glm\gnl
\gvac{1}\gob{1}{M}
\gend
\quad
\text{resp.}
\quad
\mu^+ \,
=
\gbeg{2}{3}
\got{1}{M}\got{1}{A}\gnl
\grm\gnl
\gob{1}{M}
\gend
$$
Similar, for a left $C$-comodule $N\in {}^C\C$ and a right $C$-comodule $N\in\C^C$ we use the following notation
$$
\chi^- \,
=
\gbeg{2}{3}
\gvac{1}\got{1}{N}\gnl
\glcm\gnl
\gob{1}{C}\gob{1}{N}
\gend
\quad
\text{resp.}
\quad
\chi^+ \,
=
\gbeg{2}{3}
\got{1}{N}\gnl
\grcm\gnl
\gob{1}{N}\gob{1}{C}
\gend
$$
Finally, if $B$ is a Hopf algebra in $\C$, we denote the antipode $S$ (and its inverse, if it exists) by
$$
S \,
=
\gbeg{1}{5}
\got{1}{B}\gnl
\gcl{1}\gnl
\gmp{+}\gnl
\gcl{1}\gnl
\gob{1}{B}
\gend
\quad
\text{and}
\quad
S\inv \,
=
\gbeg{1}{5}
\got{1}{B}\gnl
\gcl{1}\gnl
\gmp{-}\gnl
\gcl{1}\gnl
\gob{1}{B}
\gend
$$
Let $B$ be a Hopf algebra in $\C$. The cotensor product $M \square_B N$ of a right $B$-comodule $M$ and a left $B$-comodule $N$ in $\C$ is defined as the equalizer of $\chi^+_M \otimes N$ and $M \otimes \chi^-_N$
\[
M \square_B N \longrightarrow M \otimes N \; \substack{\longrightarrow \\ \longrightarrow} \; M \otimes B \otimes N
\]
Let $A$ be a $B$-comodule algebra in $\C$ (say with comodule structure denoted by $\chi$). The coinvariant subobject $A^{coB}$ of $A$ is given by the equalizer of $\chi$ and $A\otimes \eta_B$
\[
A^{coB} \stackrel{\iota}{\longrightarrow} A \; \substack{\longrightarrow \\ \longrightarrow} \; A \otimes B
\]
An object $X$ in $\C$ is called flat if tensoring with $X$ preserves equalizers. We say $X$ is faithfully flat if tensoring with $X$ reflects isomorphisms.
\begin{definition}\hide{[\cite{schauenburgbr1,schauenburgbr2}]}
Let $A$ be a right $B$-comodule algebra in $\C$. $A$ is said to be a right $B$-Galois object in $\C$ if $\eta_A : I \rightarrow A$ is the equalizer of $\chi$ and $A\otimes \eta_B$ (i.e., $A^{coB} = I$), and the canonical morphism $\canbr = (\nabla_A \otimes B)(A\otimes \chi^+) : A \otimes A \rightarrow A \otimes B$ is an isomorphism. Similarly we can define left $B$-Galois objects.
\end{definition}
Let $A$ be a right $B$-Galois object, let's denote
\[
\gammar = (B \overset{\eta_A \otimes B}\longrightarrow A \otimes B \overset{(\canbr)\inv}\longrightarrow A \otimes A)
\]
The morphism $\gammar$, being a partial inverse to $\canbr$, satisfies several identities, which we will list in the lemma below. Proofs can be found in \cite[Remark 3.2 and Lemma 3.4]{schauenburgbr1}.

\begin{lemma}\label{lemmagammaproperties}Let $A$ be a right $B$-Galois object in $\C$. Then
\begin{align*}
\gbeg{4}{6}
\got{3}{B}\gnl
\gvac{1}\gcl{1}\gnl
\glmpb\gnot{\gamma}\gcmpt\grmpb\gnl
\gcl{1}\gvac{1}\grcm\gnl
\gwmu{3}\gcl{1}\gnl
\gob{3}{A}\gob{1}{A}
\gend
=
\gbeg{2}{5}
\gvac{1}\got{1}{B}\gnl
\gvac{1}\gcl{1}\gnl
\gu{1}\gcl{1}\gnl
\gcl{1}\gcl{1}\gnl
\gob{1}{A}\gob{1}{B}
\gend
&\quad , \quad 
\gbeg{3}{5}
\got{1}{A}\gnl
\grcm\gnl
\gcl{1}\gnotc{\gamma}\glmptb\grmpb\gnl
\gmu\gcl{1}\gnl
\gob{2}{A}\gob{1}{A}
\gend
=
\gbeg{2}{5}
\gvac{1}\got{1}{A}\gnl
\gvac{1}\gcl{1}\gnl
\gu{1}\gcl{1}\gnl
\gcl{1}\gcl{1}\gnl
\gob{1}{A}\gob{1}{A}
\gend
\\
\gbeg{4}{5}
\got{3}{B}\gnl
\gvac{1}\gcl{1}\gnl
\glmpb\gnot{\gamma}\gcmpt\grmpb\gnl
\gcl{1}\gvac{1}\grcm\gnl
\gob{1}{A}\gvac{1}\gob{1}{A}\gob{1}{B}
\gend
=
\gbeg{4}{5}
\gvac{2}\got{1}{B}\gnl
\gvac{1}\gwcm{3}\gnl
\glmpb\gnot{\gamma}\gcmpt\grmpb\gcl{1}\gnl
\gcl{1}\gvac{1}\gcl{1}\gcl{1}\gnl
\gob{1}{A}\gvac{1}\gob{1}{A}\gob{1}{B}
\gend
&\quad , \quad
\gbeg{3}{6}
\got{3}{B}\gnl
\gvac{1}\gcl{1}\gnl
\glmpb\gnot{\gamma}\gcmpt\grmpb\gnl
\grcm\gcl{1}\gnl
\gibr\gcl{1}\gnl
\gob{1}{B}\gob{1}{A}\gob{1}{A}
\gend
=
\gbeg{4}{6}
\got{3}{B}\gnl
\gvac{1}\gcl{1}\gnl
\gwcm{3}\gnl
\gmp{+}\glmpb\gnot{\gamma}\gcmpt\grmpb\gnl
\gcl{1}\gcl{1}\gvac{1}\gcl{1}\gnl
\gob{1}{B}\gob{1}{A}\gvac{1}\gob{1}{A}
\gend
\\
\gbeg{3}{5}
\got{1}{B}\gvac{1}\got{1}{B}\gnl
\gwmu{3}\gnl
\glmpb\gnot{\gamma}\gcmpt\grmpb\gnl
\gcl{1}\gvac{1}\gcl{1}\gnl
\gob{1}{A}\gvac{1}\gob{1}{A}
\gend
=
\gbeg{4}{7}
\got{1}{B}\gvac{2}\got{1}{B}\gnl
\gcl{1}\gvac{2}\gcl{1}\gnl
\gnotc{\gamma}\glmptb\grmpb\gnotc{\gamma}\glmpb\grmptb\gnl
\gcl{1}\gbr\gcl{1}\gnl
\gbr\gmu\gnl
\gmu\gcn{2}{1}{2}{2}\gnl
\gob{2}{A}\gob{2}{A}
\gend
&\quad , \quad
\gbeg{3}{4}
\gvac{1}\gu{1}\gnl
\glmpb\gnot{\gamma}\gcmpt\grmpb\gnl
\gcl{1}\gvac{1}\gcl{1}\gnl
\gob{1}{A}\gvac{1}\gob{1}{A}
\gend
=
\gbeg{2}{4}
\gnl
\gu{1}\gu{1}\gnl
\gcl{1}\gcl{1}\gnl
\gob{1}{A}\gob{1}{A}
\gend
\end{align*}\hide{
$$
\gbeg{3}{5}
\gvac{1}\got{1}{B}\gnl
\gvac{1}\gcl{1}\gnl
\glmpb\gnot{\gamma}\gcmpt\grmpb\gnl
\gwmu{3}\gnl
\gvac{1}\got{1}{A}
\gend
=
\gbeg{2}{5}
\gvac{1}\got{1}{B}\gnl
\gvac{1}\gcl{1}\gnl
\gu{1}\gcu{1}\gnl
\gcl{1}\gnl
\got{1}{A}
\gend
$$
}
\end{lemma}

\hide{
Note that \eqref{lemmagamma1} and \eqref{lemmagamma2} are equivalent with saying that the inverse of $can^+$ is given by
$$(can^+)\inv\;
=
\gbeg{3}{5}
\got{1}{A}\got{1}{H}\gnl
\gcl{1}\gcl{1}\gnl
\gcl{1}\gnotc{\gamma}\glmptb\grmpb\gnl
\gmu\gcl{1}\gnl
\gob{2}{A}\gob{1}{A}
\gend
$$
}

\subsection{Module categories}

Following \cite{ostrik} we recall the definition of module categories.
\begin{definition}\label{defmodcat}
Let $\C$ be a monoidal category.  A \emph{left module category over $\C$}\index{module category} is a category $\M$ equipped with
\begin{itemize}
\item a bifunctor $\ast : \C \times \M \ra \M,\ (X,M) \mapsto X \ast M$,
\item natural associativity isomorphisms $m_{X,Y,M} : (X\ot Y)\ast M \ra X\ast (Y\ast M)$,
\item unit isomorphisms $l_M : I \ast M \ra M$ such that
\end{itemize}
\begin{diagram}
((X\ot Y)\ot Z)\ast M & \rTo^{a_{X,Y,Z}\ast M} & (X\ot (Y\ot Z))\ast M \\
 & & \dTo_{m_{X,Y\ot Z,M}} \\
\dTo^{m_{X\ot Y,Z,M}} & &  X \ast ((Y\ot Z) \ast M) \\
& & \dTo_{X \ast m_{Y,Z,M}} \\
(X\ot Y)\ast (Z\ast M) & \rTo^{m_{X,Y,Z\ast M}} & X\ast (Y\ast (Z\ast M))
\end{diagram} 
\begin{diagram}[nohug]
(X\ot I)\ast M & & \rTo^{m_{X,I,M}} & & X \ast (I \ast M)	\\
& \rdTo_{r_X \ast M} & & \ldTo_{X\ast l_M} & \\
& & X \ast M & &
\end{diagram}
are commutative diagrams for $X,Y,Z \in \C$ and $M\in \M$.

Equivalently, $\M$ is left module category over $\C$ if there is given a monoidal functor $\C \ra End(\M)$, where $End(\M)$ is the monoidal category of endofunctors of $\M$ (product is given by composition of functors).

\emph{Right $\C$-module categories} can be defined similarly. Let $\D$ be another monoidal category. $\M$ is said to be a \emph{$(\C,\D)$-bimodule}\index{module category!bimodule category} category if $\M$ is simultaneously a left $\C$-module category and right $\D$-module category, together with natural isomorphisms $\gamma_{X,M,Y} : (X\ast M)\ast Y \ra X\ast (M\ast Y)$ satisfying certain compatibility axioms, cf. \cite[Proposition 2.12]{greenough}.

A \emph{$\C$-module functor}\index{module category!module functor} between left $\C$-module categories $\M$ and $\N$ is a pair $(F,\theta)$, where $F : \M \rightarrow \N$ is a functor and $\theta : F(- \ast -) \ra - \ast F(-)$ is a natural isomorphism such that the following diagrams commute
\begin{diagram}
F((X \ot Y)\ast M) & \rTo^{F(m_{X,Y,M})} & F(X\ast (Y\ast M)) \\
 & & \dTo_{\theta_{X,Y\ast M}} \\
\dTo^{\theta_{X\ot Y,M}} & &  X \ast F(Y\ast M) \\
& & \dTo_{X \ast \theta_{Y,M}} \\
(X\ot Y)\ast F(M) & \rTo^{m_{X,Y,F(M)}} & X\ast (Y\ast F(M))
\end{diagram} 
\begin{diagram}[nohug]
F(I \ast M) & & \rTo^{F(l_M)} & & X \ast F(M)	\\
& \rdTo_{\theta_{I,M}} & & \ruTo_{l_{F(M)}} & \\
& & I \ast F(M) & &
\end{diagram}
for $X,Y \in \C$ and $M\in \M$.
\end{definition}

\subsection{The Brauer group}

Let $H$ be a $k$-Hopf algebra. Denote by $\yd$ the category of left-left Yetter-Drinfel'd modules. It is braided via
\begin{align*}
&\phi( m \ot n ) = m\hm{-1} \cdot n \ot m\hm{0}	\\
&\phi\inv (n \ot m) = m\hm{0} \ot S\inv (m\hm{-1}) \cdot n
\end{align*}
for $M,N \in \yd$, $m \in M$ and $n\in N$.
\\
A $k$-Hopf algebra $H$ is said to be \emph{quasitriangular} if there exists an element $R = \sum R^1 \ot R^2 \in H \ot H$, called the \emph{$R$-matrix}\index{quasitriangular!$R$-matrix}, satisfying the following conditions
\begin{align}
&\sum \va(R^1) R^2 = \sum R^1 \va(R^2) = 1,	\label{qt1}\tag{QT1}\\
&\sum \Delta(R^1) \ot R^2 = \sum R^1 \ot r^1 \ot R^2r^2,	\label{qt2}\tag{QT2}\\
&\sum R^1 \ot \Delta(R^2) = \sum R^1r^1 \ot r^2 \ot R^2,	\label{qt3}\tag{QT3}\\
&\sum R^1 h_1 \ot R^2 h_2 = \sum h_2 R^1 \ot h_1 R^2 \label{qt4}\tag{QT4}
\end{align}
for all $h \in H$, where $r=R$. From now on, we will usually omit the sum sign and write $R = R^1 \ot R^2$.

Consider the category of $H$-modules $\HM$. It is monoidal via the diagonal action (as in 2.). Furthermore, $\HM$ is braided via:
\begin{align*}
&\psi(m\ot n) = R^2\cdot n \ot R^1\cdot m	\\
&\psi\inv(n \ot m) = S(R^1)\cdot m \ot R^2\cdot n
\end{align*}
for $M,N \in \HM$, $m\in M$ and $n \in N$.

For any $M \in \HM$, there are 2 ways to define a left-left Yetter-Drinfeld module structure on $M$, by using the $R$-matrix or its inverse:
\begin{equation}\label{deflambda1en2}
\begin{split}
&\lambda_1(m) = R^2 \ot R^1 \cdot m	\\
&\lambda_2(m) = SR^1 \ot R^2 \cdot m
\end{split}
\end{equation}
In this way we can consider 2 monoidal subcategories of $\yd$, say $\mr$ resp. $\mri$. As braided monoidal categories we get 
\begin{align*}
&(\mr,\psi) \hookrightarrow (\yd,\phi)	\\
&(\mri,\psi) \hookrightarrow (\yd,\phi\inv)
\end{align*}

Now let us recall the definition of the Brauer group of a strict closed monoidal category $\C$. An Algebra in $\C$ is called an Azuamaya algebra if the following two functors are tensor equivalence functors:
\def\Aol{\overline{A}}
$$\begin{diagram}
\C & \rTo& _{A\otimes\Aol}\C, & \ \ X  \mapsto  A\ot X\\
\C & \rTo& \C_{\Aol\otimes A}, &\  \ X  \mapsto  X\ot A
\end{diagram}$$
where $\overline{A}$ is the opposite algebra of $A$ in $\C$. For more detail, the reader is referred to \cite{vanoystaeyenzhangbraided}.

\subsection{Transmutation theory}

It is well-known that $H$ can be deformed into a braided Hopf algebra $\tr$ via Majid's transmutation process \cite{majidtransmutation}. In particular, $\tr$ equals $H$ as an algebra. It becomes an $H$-module algebra with action given by
\[
h \rhd x = h_1 x S(h_2)
\]
for $h,x\in H$. One can turn $\tr$ into a braided Hopf algebra in $\HM$ as follows: the counit is the same as $\va_H$, the comultiplication and antipode are given by
\begin{align*}
\ud (x) &= x_1 S(R^2) \ot R^1 \rhd x_2	\\
		&= x_1 S(r^2)S(R^2) \ot R^1 x_2 S(r^1)	\\
		&= x_1 r^2S(R^2) \ot R^1 x_2 r^1	\\
		&= r^2x_2 S(R^2) \ot R^1 r^1 x_1 	\\
		&= R^2 \rhd x_2 \ot R^1 x_1
\intertext{and}
\uS (x) &= R^2 S(R^1\rhd x)
\end{align*}
for $x\in H$.

As $\tr=H$ as algebra and since
\[
(h_1 \rhd x)\cdot (h_2\cdot m) = h\cdot(x\cdot m)
\]
for
$h,x\in H$ and $m\in M$, where $M\in \HM$, every $H$-module is naturally an $\tr$-module. Let $\mO$ be the class of $\tr$-modules obtained in this canonical way. Then $(\ud,\mO)$ is an opposite comultiplication in the sense of \cite{majidbraidedgroups}. Furthermore, $(\tr,\ud,\uR=1\ot1)$ is a quasitriangular in the category $\HM$ (see \cite[Definition 1.3]{majidtransmutation} or \cite[Section 4]{majidreconstruction}). As observed in \cite{ZZ}, $\tr$ is flat in $\HM$. Finally, the braided Hopf algebra $\tr$ is cocommutative cocentral, in the sense of \cite{schauenburgbr2}, the half-braiding is defined by
\[
\sigma_{\tr,M} : \tr \ot M \ra M \ot \tr,\ \sigma_{\tr,M} (x \ot m) = r^2R^1 \cdot m \ot r^1xR^2
\]
for $m \in M$, $M \in \HM$ and $x\in H$. Following \cite{schauenburgbr2}, we say a bicomodule $M$ in ${}^{\tr}(\HM)^{\tr}$ \emph{cocommutative}\index{comodule!bicomodule!cocommutative} if 
\begin{align*}
\chi^+=\sigma_{\tr,M}\circ\chi^- = (M \overset{\chi^-}\longrightarrow \tr \ot M \overset{\sigma_{\tr,M}}\longrightarrow M \ot \tr)
\end{align*}
\hide{We will use $\trbicom$ to denote the category of cocommutative $\tr$-bicomodules in $\HM$.} 


\section{Bi-Galois objects versus tensor equivalences}\label{secbrequiv}

In this section, we will investigate the relation between tensor equivalences and braided bi-Galois objects. We are inspired by the following result due to \mbox{Schauenburg}.
\begin{proposition}[{\cite[Corollary 5.7]{schauenburg2}}]\label{propschauenburg}
Let $k$ be a commutative ring and let $B$ and $L$ be $k$-flat Hopf algebras. The following are equivalent:
\begin{itemize}
\item[1.] ${}^B\M$ and ${}^L\M$ are equivalent as monoidal $k$-linear categories ($B$ and $L$ are then said to be \emph{monoidally co-Morita equivalent}),
\item[2.] there is faithfully flat $L$-$B$-bi-Galois extension of $k$.
\end{itemize}
\end{proposition}

Let $(\C,\phi)$ be a braided (strict) monoidal category with equalizers. Suppose $B$ and $L$ are flat Hopf algebras in $\C$. The `if' statement of the aforementioned proposition has been generalized to the braided setting by Schauenburg in \cite{schauenburgbr1,schauenburgbr2}, namely, if $A$ is a faithfully flat $L$-$B$-bi-Galois object, then the cotensor functor $\alpha_A = A \ct_B -: \BC \rightarrow \LC$ is a tensor equivalence. In particular, for two $B$-comodules $M$ and $N$, there's an isomorphism
\[
\xi : (A \ct_B M) \ot (A \ct_B N) \rightarrow A \ct_B (M \ot N)
\]
which is induced by $\xi_0 = (\nabla_A \ot M \ot N) \circ (A \ot \phi_{A,M} \ot N) \circ (\iota \ot \iota)$. Graphically, $\xi$ satisfies
\begin{align*}
\gbeg{4}{7}
\got{1}{\aam}\gvac{1}\got{1}{\phantom{aa}\alpha_A(N)}\gnl
\gcl{1}\gvac{1}\gcl{1}\gnl
\glmpt\gnot{\xi}\gcmpb\grmpt\gnl
\gvac{1}\gcl{1}\gnl
\glmpb\gnot{\iota}\gcmpt\grmp\gnl
\gcl{1}\gcn{2}{1}{2}{2}\gnl
\gob{1}{A}\gob{2}{M\ot N}
\gend
=
\;
\gbeg{4}{6}
\got{2}{\aam}\gvac{0}\got{2}{\phantom{aa}\alpha_A(N)}\gnl
\gcl{1}\gvac{1}\gcl{1}\gnl
\gnotc{\iota}\glmptb\grmpb\gnotc{\iota}\glmptb\grmpb\gnl
\gcl{1}\gbr\gcl{1}\gnl
\gmu\gcl{1}\gcl{1}\gnl
\gob{2}{A}\gob{1}{M}\gob{1}{N}
\gend
\end{align*}

The category $\BC$ is naturally a right $C$-module category. Indeed, if $M \in \BC$ and $X \in \C$, we can define $M \ast X = M \ot X$, which is the tensor product in $\C$ with left $B$-coaction given by $\chi^-_M \ot X$.

Any object $X \in \C$ can be seen as an object in $\BC$ if we equip $X$ with the trivial left $B$-comodule structure $\eta_B \ot X$. We will denote this comodule by $X^t$, although sometimes we will just write $X$ if the situation will make clear that $X \in \C$ is equipped with the trivial comodule structure.

\begin{definition}[\cite{davydovnikshych}]\label{deftriv}
Let $\C$ and $\D$ be monoidal categories and suppose $\E$ is a monoidal subcategory of both $\C$ and $\D$. A monoidal quivalence $\alpha : \C \ra \D$ is said to be \emph{trivializable on $\E$}\index{tensor equivalence!trivializable} if the restriction $\alpha|_{\E}$ is isomorphic to $id_{\E}$ as monoidal functors.

We will denote by $\Aut(\C)$ respectively $\Aut(\C,\E)$ the group of isomorphism classes of monoidal autoequivalences of $\C$, respectively monoidal autoequivalences of $\C$ \mbox{trivializable} on $\E$. If $\C$ and $\E$ are braided, we denote by $\Aut^{br}(\C,\E)$ the group of isomorphism classes of braided monoidal autoequivalences of $\C$ trivializable on $\D$.
\end{definition}

\begin{lemma}\label{lemmatriviscmod}
Let $B$ and $L$ be flat Hopf algebras in $\C$ and suppose $\alpha : \BC \rightarrow \LC$ is a (strong) monoidal functor. Then $\alpha$ is trivializable on $\C$ if and only if $\alpha$ is a right $\C$-module functor.
\end{lemma}
\begin{proof}
Suppose $\alpha$ is a right $\C$-module functor. The unit object $I \in \C$ can also be seen as an object in $\BC$ (with trivial $B$-comodule structure). Then
\[
\alpha(X) \cong \alpha(I \ot X^t) \overset{\theta_{I,X}}{\cong} \alpha(I)\ot X \cong I \ot X \cong X
\]
for any $X \in \C$. Conversely, suppose $\alpha$ is trivializable on $\C$, then
\[
\alpha(M \ot X) \overset{\varphi\inv_{M,X}}\cong \am \ot \alpha(X) \cong \am \ot X
\]
for $M\in \BC$ and $X \in \C$.
\end{proof}

\begin{lemma}\label{lemmaactistriv}
Let $B$ and $L$ be flat Hopf algebras in $\C$. Suppose $A$ is a faithfully flat $L$-$B$-bi-Galois object. The tensor equivalence functor $A \ct_B - : \BC \ra \LC$ is trivializable on $\C$, or equivalently, $\alpha$ is a right $\C$-module functor.
\end{lemma}
\begin{proof}
For any $M \in \BC$ we have
\begin{align*}
\begin{diagram}[size=2em,nohug]
1	&	\rTo	&	A \ot (A \ct_B M)		& \rTo^{}	 		& A \ot (A \ot M) 			& \pile{\rTo^{}\\\rTo_{}}					& A \ot (A \ot B \ot M)  \\
	&			&	\dTo_{\sim}		& 			& \dTo_{=}						& 		& \dTo_{=} \\
1	&	\rTo	&	(A \ot A) \ct_B M		& \rTo^{}	 		& (A \ot A) \ot M 			& \pile{\rTo^{}\\\rTo_{}}					& (A \ot A) \ot B \ot M  \\
	&			&	\dTo_{\sim}		& 			& \dTo_{can_+ \ot M}						& 		& \dTo_{=} \\
1	&	\rTo	&	(A \ot B) \ct_B M		& \rTo^{}	 		& (A \ot B) \ot M 			& \pile{\rTo^{}\\\rTo_{}}					& (A \ot B) \ot B \ot M  \\
	&			&	\dTo_{\sim}		& 			& \dTo_{=}						& 		& \dTo_{=} \\
1	&	\rTo	&	A \ot (B \ct_B M)		& \rTo^{}	 		& A \ot (B \ot M) 			& \pile{\rTo^{}\\\rTo_{}}					& A \ot (B \ot B \ot M)  \\
	&			&	\dTo_{\sim}		& 			&	\\
	&			&	A \ot M
\end{diagram}
\end{align*}
The first and fourth sequence are exact because $A$ is flat. The associativity constraints are identities, as $\C$ is assumed to be strict. Hence $A \ot M \cong A \ot (A \ct_B M)$, where the isomorphism $A \ot M \ra A \ot (A \ct_B M)$ is induced by the morphism
\begin{align}\label{eqinduced}
\gbeg{4}{5}
\got{1}{A}\gvac{2}\got{1}{M}\gnl
\gcl{1}\gvac{1}\glcm\gnl
\gcl{1}\gnotc{\gamma}\glmpb\grmptb\gcl{1}\gnl
\gmu\gcl{1}\gcl{1}\gnl
\gob{2}{A}\gob{1}{A}\gob{1}{M}
\gend
\end{align}
Now let $X \in \C$ and consider $X^t$. The morphism $\eta_A \ot X: X \ra A \ot X$ induces a morphism, say $f : X \ra A \ct_B X$. Moreover if $X$ has trivial $B$-comodule structure \eqref{eqinduced} becomes
\begin{align*}
\gbeg{4}{5}
\got{1}{A}\gvac{2}\got{1}{X^t}\gnl
\gcl{1}\gvac{1}\glcm\gnl
\gcl{1}\gnotc{\gamma}\glmpb\grmptb\gcl{1}\gnl
\gmu\gcl{1}\gcl{1}\gnl
\gob{2}{A}\gob{1}{A}\gob{1}{X^t}
\gend
=
\gbeg{4}{5}
\got{1}{A}\gvac{2}\got{1}{X^t}\gnl
\gcl{1}\gvac{1}\gu{1}\gcl{1}\gnl
\gcl{1}\gnotc{\gamma}\glmpb\grmptb\gcl{1}\gnl
\gmu\gcl{1}\gcl{1}\gnl
\gob{2}{A}\gob{1}{A}\gob{1}{X^t}
\gend
\overset{\text{(Lemma \ref{lemmagammaproperties})}}
=
\gbeg{4}{4}
\got{1}{A}\gvac{2}\got{1}{X^t}\gnl
\gcl{1}\gu{1}\gu{1}\gcl{1}\gnl
\gmu\gcl{1}\gcl{1}\gnl
\gob{2}{A}\gob{1}{A}\gob{1}{X^t}
\gend
=
\gbeg{3}{4}
\got{1}{A}\gvac{1}\got{1}{X^t}\gnl
\gcl{1}\gu{1}\gcl{1}\gnl
\gcl{1}\gcl{1}\gcl{1}\gnl
\gob{1}{A}\gob{1}{A}\gob{1}{X^t}
\gend
\end{align*}
Hence the isomorphism $A \ot X \cong A \ot (A \ct_B X)$ coincides with $A \ot f$. By faithfully flatness of $A$, $f$ must be an isomorphism in $\C$. Thus $X \cong A \ct_B X$ (as $\C$-objects).
\end{proof}
Consider $\alpha_A = A \ct_B -$ as in the previous lemma. Let $\U : \LC \rightarrow \C$ be the forgetful functor and define $\omega_A = \U \circ \alpha_A : \LC \rightarrow \C$. Thus if $M \in \BC$, then $\aam=\oam$ as $\C$-objects, so if we want to emphasize the fact that we treat $\aam$ as a $\C$-object, we can (but not always will) use $\oam$.
\\
The tensor product of two $B$-comodules in $\C$ is again a $B$-comodule through the diagonal coaction. In particular, if $X^t \in \C$ and $M \in \BC$ arbitrary, then $X^t \ot M \in \BC$, then
\begin{align}
\chi^-_{X^t\ot M}
=
\gbeg{3}{4}
\got{1}{X^t}\gvac{1}\got{1}{M}\gnl
\gcl{1}\glcm\gnl
\gbr\gcl{1}\gnl
\gob{1}{B}\gob{1}{X^t}\gob{1}{M}\gnl
\gend\notag
\intertext{By the the naturality}
\label{eqbraidinghcolinear}
\gbeg{3}{4}
\got{1}{X^t}\gvac{1}\got{1}{M}\gnl
\gcl{1}\glcm\gnl
\gbr\gcl{1}\gnl
\gob{1}{B}\gob{1}{X^t}\gob{1}{M}\gnl
\gend
=
\gbeg{3}{5}
\gvac{1}\got{1}{X^t}\got{1}{M}\gnl
\gvac{1}\gbr\gnl
\glcm\gcl{1}\gnl
\gcl{1}\gibr\gnl
\gob{1}{B}\gob{1}{X^t}\gob{1}{M}\gnl
\gend
\end{align} 
which is saying that the braiding $\phi_{X^t,M}:X^t\ot M \ra M \ot X^t$ is a morphism in $\BC$. We can now make the following observation.
\begin{lemma}\label{lemmaalphaaa}
With notation as above, we have
\begin{align*}
\begin{diagram}
\omega_A(X^t) \ot \oam	&	\rTo^{\varphi_{X,M}}	&	\omega_A(X^t \ot M)	\\
\dTo^{\phi_{\omega_A(X^t),\oam}}			&					&	\dTo_{\omega_A(\phi_{X^t,M})}	\\
\oam \ot \omega_A(X^t)	&	\rTo^{\varphi_{M,X}}	&	\omega_A(M \ot X^t)	\\
\end{diagram}\tag{$\mathbb{A}$}
\end{align*}
for $M\in \BC$ and $X \in \C$.
\end{lemma}
\begin{proof}
Let $f : X^t \ra A \ct_B X^t$ be the isomorphism in $\C$ induced by $\eta_A \ot X$ as in Lemma \ref{lemmaactistriv}. The morphism $\omega_A(\phi_{M,X^t}) \circ \varphi_{X,M} \circ (f \ot (A \ct_B M)) : X^t \ot (A \ct_B M) \ra A \ct_B (M \ot X^t)$ is induced by
\begin{align*}
(A \ot \phi_{M,X^t}) \circ \xi_0 \circ (\eta_A \ot X)
&=
\gbeg{4}{5}
\gvac{1}\got{1}{X}\got{1}{A}\got{1}{M}\gnl
\gu{1}\gbr\gcl{1}\gnl
\gmu\gcl{1}\gcl{1}\gnl
\gcn{2}{1}{2}{2}\gbr\gnl
\gob{2}{A}\gob{1}{M}\gob{1}{X}
\gend
\intertext{while the morphism $\varphi_{M,X}\circ \phi_{\omega_A(X^t),\oam} \circ (f \ot (A \ct_B M))$ is induced by}
\xi_0 \circ \phi_{A\ot X^t,A\ot M} \circ (\eta_A \ot X)
&=
\gbeg{4}{7}
\gvac{1}\got{1}{X}\got{1}{A}\got{1}{M}\gnl
\gu{1}\gbr\gcl{1}\gnl
\gbr\gbr\gnl
\gcl{1}\gbr\gcl{1}\gnl
\gcl{1}\gbr\gcl{1}\gnl
\gmu\gcl{1}\gcl{1}\gnl
\gob{2}{A}\gob{1}{M}\gob{1}{X}
\gend
\end{align*}
As both diagrams are equal by the the naturality and since $f \ot (A \ct_B M)$ is an isomorphism, we obtain
$\omega_A(\phi_{X^t,M}) \circ \varphi_{X,M} = \varphi_{M,X}\circ \phi_{\omega_A(X^t),\oam}$.
\end{proof}
Thus a faithfully flat braided $L$-$B$-bi-Galois object $A$ induces a monoidal (right $\C$-linear) equivalence $A \ct_B - : \BC \ra \LC$ which is trivializable on $\C$ and satisfies $(\mathbb{A})$. Our next goal is to investigate whether the converse statement is valid. That is, suppose $\alpha : \BC \ra \LC$ is a tensor equivalence trivializable on $\C$ and satisfying $(\mathbb{A})$, does $\alpha$ come from a faithfully flat bi-Galois object? By Lemma \ref{lemmatriviscmod}, $\alpha$ is a right $\C$-module functor with
\begin{align}\label{defthetamx}
\theta_{M,X} = \big( \alpha(M \ot X) \overset{\varphi\inv_{M,X}}\cong \am \ot \alpha(X) \cong \am \ot X \big)
\end{align}
for $M\in \BC$ and $X \in \C$.
\\\\
Our approach is inspired by \cite{ulbrich1}, in which the author assigns to a fibre functor $\omega : {}^H\M \ra {}_k\M$, the right $H$-Galois object $\omega(H)$ (here $H$ is a $k$-Hopf algebra). 
\\
Let us  denote $\omega = \U \circ \alpha : \BC \ra \C$, where $\U : \LC \ra \C$ is the forgetful functor. We can use $\omega$ if we want to emphasize that we're working on the level of $\C$-objects. For example, we can say that $\alpha(B)$ is an algebra (in $\LC$), or equivalently, $\omega(B)$ is an $L$-comodule algebra in $\C$.
\\\\
Suppose $M$ is an algebra in $\BC$. It is known that a monoidal functor sends algebras to algebras.  Hence, $\am \in \LC$ is an algebra, or equivalently, $\om$ is a left $L$-comodule algebra in $\C$. As an algebra in $\C$, $\om$ has multiplication map
\begin{align}\label{defnablaam}
\nabla_{\om} = \big( \om \ot \om \overset{\varphi_{M,M}}\longrightarrow \omega(M\ot M) \overset{\omega(\nabla_M)}\longrightarrow \om \big)
\end{align}
and unit
\[
I \cong \omega(I) \overset{\omega(\eta_M)}\longrightarrow \om
\]

\hide{
To show that $\am \in \LC$ is an algebra, first observe that by the left comodule algebra property,
we have that $\nabla_M$ is $B$-colinear in $\C$, i.e. $\nabla_M \in \BC$. Thus $\alpha(\nabla_M)\in \LC$. Now consider the following diagram.
\begin{align*}
\begin{diagram}
\am \ot \am & \rTo^{\varphi_{M,M}} & \alpha(M \ot M) 			& \rTo^{\alpha(\nabla_M)}					& \am \\
\dTo_{} 	& {\text{\small (I)}}		& \dTo_{}	& 		{\text{\small (II)}}				& \dTo_{} \\
L \ot \am \ot \am & \rTo^{L \ot \varphi_{M,M}} & L \ot \alpha(M \ot M)		& \rTo^{L \ot \alpha(\nabla_M)}				& L \ot \am
\end{diagram}
\end{align*}
where the vertical maps are given by the respective comodule morphisms. Diagram (I) commutes as $\varphi_{M,M}$ is a morphism in $\LC$, while (II) commutes since $\alpha(\nabla_M)$ is a morphism in $\LC$, as observed above. By commutativity of the diagram, we obtain that $\om$ is a left $L$-comodule algebra.}

Suppose $F$ is another (flat) Hopf algebra in $\C$. Let $M$ be a $B$-$F$-bicomodule. By the bicomodule property, the comodule structure $\chi^+_M$ can be seen as a left $B$-colinear morphism $M \ra M \ot F^t$. We can now define a right $F$-comodule structure on $\om$ as follows
\[
\om \overset{\omega(\chi^+_M)}\longrightarrow \omega(M\ot F^t) \overset{\theta_{M,F^t}}\longrightarrow \om \ot F
\]
We will now prove that if $M$ is a $B$-$F$-bicomodule algebra, then $\om$ is a right $F$-comodule algebra, i.e.
\begin{align*}
\gbeg{3}{4}
\got{1}{\om}\gvac{1}\got{1}{\om}\gnl
\gwmu{3}\gnl
\gvac{1}\grcm\gnl
\gob{2}{\phantom{a}\om}\gob{1}{F}
\gend
\;
=
\;
\gbeg{4}{5}
\got{1}{\om}\gvac{1}\got{1}{\om}\gnl
\grcm\grcm\gnl
\gcl{1}\gbr\gcl{1}\gnl
\gmu\gmu\gnl
\gob{2}{\om}\gob{2}{F}
\gend
\end{align*}
that is, by definition of $\nabla_{\om}$ and $\chi^+_{\om}$, we want the outer diagram of the following diagram to commute
\begin{align}\label{diaomcomodalg}
\begin{diagram}[labelstyle=\scriptstyle]
\om\ot\om	&		\rTo^{\varphi_{M\ot M}}		&		\omega(M\ot M)		&		\rTo^{\omega(\nabla_M)}		&		\om		\\
\dTo_{\omega(\chi^+_M)\ot\omega(\chi^+_M)}	&	(I)				&		\dTo_{\omega(\chi^+_M\ot\chi^+_M)}	&	(II)		&				\\
\omega(M\ot F^t)\ot\omega(M\ot F^t)		&		\rTo^{\varphi_{M\ot F^t,M\ot F^t}}		&		\omega(M\ot F^t\ot M \ot F^t)		&				&		\dTo^{\omega(\chi^+_M)}		\\
\dTo_{\theta_{M,X}\ot\theta_{M,X}}	&					&		\dTo_{\omega(M \ot \phi_{F^t,M} \ot F)}	&					&				\\
\om\ot F \ot \om \ot F		&			(IV)		&		\omega(M\ot M\ot F^t\ot F^t)		&		\rTo^{\omega(\nabla_M\ot\nabla_F)}		&		\omega(M\ot F^t)		\\
\dTo_{\om\ot \phi_{F,\om} \ot F}	&					&		\dTo_{\theta_{M\ot M,F\ot F}}	&		(III)			&		\dTo^{\theta_{M,F}}		\\
\om\ot\om\ot F \ot F		&		\rTo^{\varphi_{M,M}\ot F\ot F}		&			\omega(M\ot M) \ot F \ot F		&		\rTo^{\omega(\nabla_M)\ot\nabla_F}		&		\om \ot F
\end{diagram}\notag
\ \\[1ex]
\end{align}
Now (I) commutes by the the naturality of $\varphi$, (II) commutes since $M$ is assumed to be a right $F$-comodule algebra in $\C$ 
and (III) commutes by the the naturality of $\theta$. So it suffices to show the commutativity of diagram (IV). Taking the definition of $\theta_{M,X}$ as in \eqref{defthetamx} into consideration, we can divide (IV) into smaller diagrams as follows
\hide{
\begin{align*}
\begin{diagram}
A 		&	&\rTo^{}		&	&		A		&		\rTo^{}		&&	A		\\
		&	&			&	\ruTo	&		&					&\ruTo^{}			\\
\dTo^{}	&	&A			&	&		\rTo^{} &			A	&&	\dTo^{}		\\
		&	\ruTo^{}&	&	& \ruTo^{}	\\
A		&	&\rTo^{}	&	A&	&	\rTo^{}	& &A	\\
\dTo^{}		&	& 	&	\dTo^{}&	&	 	& &\dTo^{}	\\
A		&	&\rTo^{}	&	A&	&	\rTo^{}	& &A
\end{diagram}
\end{align*} 
}
\begin{landscape}
\scriptsize
\begin{align*}
\begin{diagram}	
\overset{\displaystyle \omega(M\ot F^t)}{\phantom{aa}\ot\omega(M\ot F^t) }	&	&\rTo^{\varphi}		&	&		\omega(M\ot F^t\ot M \ot F^t)		&		\rTo^{\omega(F\ot\phi\ot F)}		&&	\omega(M\ot M\ot F^t\ot F^t)		\\
		&	(i)&			&	\ruTo^{\varphi \circ id \ot \varphi}	&	(ii)	&					&\ruTo^{\varphi \circ id \ot \varphi}		& 	\\
\dTo^{\varphi\inv \ot \varphi\inv}	&	&\om\ot\omega(F^t\ot M)\ot\omega(F^t)			&	&		\rTo^{id \ot \omega(\phi) \ot id} &			\om\ot\omega(M \ot F^t)\ot\omega(F^t)	&&	\dTo^{\varphi\inv}		\\
		&	\ruTo^{id \ot \varphi \ot id \phantom{aaa}}& (iv) &	& \ruTo^{id \ot \varphi \ot id} & (iii)	\\
\overset{\displaystyle \om\ot \omega(F^t)}{\phantom{aa}\ot \om \ot\omega(F^t) }			&	&\rTo^{id \ot \phi \ot id}	&	\om\ot \om\ot\omega(F^t) \ot\omega(F^t)&	&	\rTo^{\varphi\ot\varphi}	& &\omega(M\ot M) \ot \omega(F^t \ot F^t)	\\
\dTo^{\sim}		&	& 	&	\dTo^{\sim}&	&	 	& &\dTo^{\sim}	\\
\overset{\displaystyle \om\ot F}{\phantom{aa}\ot \om \ot F }	&	&\rTo^{id \ot \phi \ot id}	&	\om\ot\om\ot F \ot F&	&	\rTo^{\varphi\ot id \ot id}	& &\omega(M\ot M) \ot F \ot F
\end{diagram}
\end{align*}
\normalsize
\end{landscape}
(i) and (iii) commute since $\omega$ is monoidal while (ii) commutes since the braiding $\phi_{F^t,M}:F^t\ot M \ra M\ot F^t$ is a morphism in $\BC$ as observed in \eqref{eqbraidinghcolinear}. Diagram (iv) commutes since we assume that the functor $\alpha$ is satisfying diagram $(\mathbb{A})$. Finally, the bottom two diagrams commute by the the naturality. Thus, $\om$ is a right $F$-comodule algebra.
\begin{proposition}\label{propahisbicomodalg}
Let $\alpha : \BC \ra \LC$ be a tensor equivalence trivializable on $\C$ and satisfying (A) and  denote by $\omega$ the composite $\U \circ \alpha : \BC \ra \C$ as before. Let $M$ be a $B$-$F$-bicomodule algebra, then $\om$ is an $L$-$F$-bicomodule algebra in $\C$.
\\
In particular, $\omega(B)$ is a flat $L$-$B$-bicomodule algebra in $\C$.
\end{proposition}
\begin{proof}
$\om$ is already shown to be a left $L$-comodule algebra and a right $F$-comodule algebra, it remains to prove that $\om$ is an $L$-$F$-bicomodule. By definition of 
$\chi^+_{\am}$, we need to prove that the following diagram is commutative.
\begin{align*}
\begin{diagram}
\om 			& \rTo^{\omega(\chi^+_M)}	 		& \omega(M \ot F^t) 			& \rTo^{\theta_{M,F^t}}					& \om \ot F \\
\dTo_{\chi^-_{\om}} 		& {\text{\small (I)}}			& \dTo_{\chi^-_{\omega(M \ot F^t) }}						& {\text{\small (II)}}			& \dTo_{\chi^-_{\om}\ot F} \\
L \ot \om		& \rTo^{L \ot \omega(\chi^+_M)}	& L \ot \omega(M \ot F^t)		& \rTo^{L \ot \theta_{M,F^t}}				& L \ot \om \ot F
\end{diagram}
\end{align*}
As mentioned before, $\chi^+_M: M \rightarrow M \ot F^t$ is left $B$-colinear, thus $\omega(\chi^+_M) \in \LC$. So diagram (I) commutes. Furthermore, diagram (II) commutes as $\theta_{M,F^t}$ is a morphism in $\LC$.
\\
Since $B$ is naturally a $B$-bicomodule algebra via its comultiplication, $\omega(B)$ becomes an $L$-$B$-bicomodule algebra in $\C$. Finally, as $\alpha$ is an equivalence, it's immediate that $\omega(B)$ is a flat object in $\C$.
\end{proof}

Let $M \in \BC$. The comodule structure $\chi^-_M$ can be seen as a $B$-colinear morphism $M \ra B \ot M^t$. It is well-known that $M \cong B \ct_B M$ as $B$-comodules in $\C$. Hence
\begin{align*}
\begin{diagram}[size=2em,nohug]
1	&	\rTo	&	M	&	\rTo^{\phantom{a}\chi^-_M \phantom{a}} &	B \ot M^t	&	\pile{\rTo^{B \ot \chi^-_M} \\ \rTo_{\Delta_B \ot M}} & B \otimes B^t \ot M^t
\end{diagram}
\end{align*}
is exact in $\BC$. As $\alpha$ is exact ($\alpha$ being an equivalence), the sequence
\begin{align*}
\begin{diagram}[size=2em,nohug]
1 	&	\rTo	&	\am &	\rTo^{\alpha(\chi^-_M)}& \alpha(B \ot M^t) &	\pile{\rTo^{\alpha(B \ot \chi^-_M)} \\ \rTo_{\alpha(\Delta_B \ot M)}}	&	\alpha(B \otimes B^t \ot M^t)
\end{diagram}
\end{align*}
is exact in $\LC$. Since $\ab$ is an $L$-$B$-bicomodule and by definition of the cotensor product $\ab \ct_B M$, the sequence
\begin{align*}
\begin{diagram}[size=2em,nohug]
1	&	\rTo	& 	\ab \ct_B M	& \rTo^{}	& \ab \ot M		& \pile{\rTo^{\alpha(B) \ot \chi^-_M}\\\rTo_{\chi^+_{\ab} \ot M}}				& \ab \ot B \ot M
\end{diagram}
\end{align*}
is also exact in $\LC$. These two sequences in $\LC$ can be linked by $\theta$ as follows
\begin{align*}
\begin{diagram}
1	&	\rTo	&	\am		& \rTo^{\alpha(\chi^-_M) \phantom{aa}}	 		& \alpha(B \ot M^t) 			& \pile{\rTo^{\alpha(B \ot \chi^-_M)}\\\rTo_{\alpha(\Delta_B \ot M)}}					& \alpha(B \otimes B^t \ot M^t) \\
	&			&	\dDashto		& 			& \dTo^{\theta_{B,M}}						& 		& \dTo_{\theta_{B,B\ot M}} \\
1	&	\rTo	& 	\ab \ct_B M	& \rTo^{}	& \ab \ot M		& \pile{\rTo^{\alpha(B) \ot \chi^-_M}\\\rTo_{\chi^+_{\ab} \ot M}}				& \ab \ot B \ot M
\end{diagram}
\end{align*}
Indeed $\theta_{B,B\ot M} \circ \alpha(B \ot \chi^-_M) = (\alpha(B) \ot \chi^-_M)\circ\theta_{B,M}$ by the the naturality of $\theta$ and 
\begin{align*}
\begin{diagram}[nohug]
\alpha(B \ot M^t) 			& \rTo^{\alpha(\Delta_B \ot M)}			& \alpha(B \otimes B^t \ot M^t) \\
\dTo_{\theta_{B,M}}				&										&\dTo_{\theta_{B\ot B^t,M}} \\
\ab \ot M					& \rTo^{\alpha(\Delta_B) \ot M}			& \alpha(B \ot B^t) \ot M	\\
							& \rdTo_{\chi^+_{\ab} \ot M}			&\dTo_{\theta_{B,B} \ot M}	\\
							&										& \ab \ot B \ot M
\end{diagram}
\end{align*}
commutes by the naturality of $\theta$ and by definition of $\chi^+_{\ab}$. Hence
\[
\am \cong \ab \ct_B M
\]
is an isomorphism in $\LC$, say $G_M$, for any $B$-comodule $M$ in $\C$. The isomorphism $G : \alpha(-) \ra \ab \ct -$ is easily seen to be natural (since $\theta$ is).
\begin{remark}
For the sake of convenience, we will no longer make a distinction between $\alpha(M)$ and $\omega(M)$, as they are the same object in $\C$. If we say, for example, that $\alpha(M)$ is a right $F$-comodule algebra, it is understood that we mean that $\alpha(M)=\omega(M)\in \C$ is a right $F$-comodule algebra in $\C$.
\end{remark}
Next we'll show that, if $M$ is a $B$-$F$-bicomodule (algebra), then $\am \cong \ab \ct_B M$ is a left $L$-colinear and right $F$-colinear (algebra) isomorphism. To show that it is right $F$-colinear, the following diagram should commute.
\begin{align*}
\begin{diagram}
\am	&	\rTo^{G_M}	&	\ab \ct_B M	\\
\dTo^{\alpha(\chi^+_M)}	& (I)	&	\dTo_{\ab \ot \chi^+_M}	\\
\alpha(M\ot F^t)	&	\rTo^{G_{M\ot F^t}}	&	\ab \ct_B (M\ot F^t)	\\
\dTo^{\theta_{M,F}}	&	(II)	&	\dTo_{\sim}	\\
\am\ot F	&	\rTo^{G_M\ot F}	&	(\ab \ct_B M) \ot F
\end{diagram}
\end{align*}
The top diagram commutes as the isomorphism $G$ is natural. To show the commutativity of the bottom diagram, observe that
\begin{align*}
\begin{diagram}[size=2em,nohug]
1	&	\rTo	&	\alpha(M\ot F^t)		& \rTo^{}	 		& \alpha(B \ot M^t \ot F^t) 			& \pile{\rTo^{}\\\rTo_{}}					& \alpha(B \ot B \ot M^t \ot F^t)  \\
	&			&	\dTo_{G_{M\ot F^t}}		& 			& \dTo_{\theta_{B,M\ot F}}						& 		& \dTo_{\theta_{B,B\ot M\ot F}} \\
1	&	\rTo	&	\ab \ct_B (M\ot F^t)		& \rTo^{}	 		& \ab \ot M \ot F^t 			& \pile{\rTo^{}\\\rTo_{}}					& \ab \ot B \ot M \ot F^t \\
	&			&	\dTo_{\sim}		& 			& \dTo_{=}						& 		& \dTo_{=} \\
1	&	\rTo	&	(\ab \ct_B M) \ot F		& \rTo^{}	 		& \ab \ot M \ot F 			& \pile{\rTo^{}\\\rTo_{}}					& \ab \ot B \ot M \ot F \\
\end{diagram}
\intertext{where the last sequence is exact since $F$ is flat, while}
\begin{diagram}[size=2em,nohug]
1	&	\rTo	&	\alpha(M\ot F^t)		& \rTo^{}	 		& \alpha(B \ot M^t \ot F^t) 			& \pile{\rTo^{}\\\rTo_{}}					& \alpha(B \ot B \ot M^t \ot F^t)  \\
	&			&	\dTo_{\theta_{M,F}}		& 			& \dTo_{\theta_{B\ot M^t,F}}						& 		& \dTo_{\theta_{B\ot B^t\ot M^t,F}} \\
1	&	\rTo	&	\am\ot F		& \rTo^{}	 		& \alpha(B\ot M^t) \ot F 			& \pile{\rTo^{}\\\rTo_{}}					& \alpha(B \ot B^t \ot M^t)\ot F \\
	&			&	\dTo_{G_M\ot F}		& 			& \dTo_{\theta_{B,M}\ot F}						& 		& \dTo_{\theta_{B,B\ot M}\ot F} \\
1	&	\rTo	&	(\ab \ct_B M) \ot F		& \rTo^{}	 		& \ab \ot M \ot F 			& \pile{\rTo^{}\\\rTo_{}}					& \ab \ot B \ot B \ot M \\
\end{diagram}
\end{align*}
where we've again used the flatness of $F$. As $\theta_{B,M\ot F} = (\theta_{B,M}\ot F) \circ \theta_{B\ot M^t,F}$ and $\theta_{B,B\ot M\ot F} = (\theta_{B,B\ot M}\ot F)\circ \theta_{B\ot B^t\ot M^t,F}$, we obtain the commutativity of (I). Thus $\ab \cong \ab \ct_B M$ as $L$-$F$-bicomodules.
\\
To show that $G_M$ is an algebra morphism, we have to show $\nabla_{\ab \ct_B M} \circ (G_M \ot G_M) = G_M \circ \nabla_{\ab}$, or 
\[
\iota \circ \nabla_{\ab \ct_B M} \circ (G_M \ot G_M) = \iota \circ G_M \circ \nabla_{\ab}
\]
by the monicity of $\iota : \ab \ct_B M \ra \ab \ot M$. Observe
\begin{align*}
&\iota \circ G_M \circ \nabla_{\ab}		\\
\by{eqaeproduct}&=	\theta_{B,M} \circ \alpha(\chi^-_M) \circ \alpha(\nabla_M) \circ \varphi_{M,M}\\
 &=	\theta_{B,M}	\circ	\alpha(\nabla_B\ot\nabla_M) \circ \alpha(B\ot \phi_{M^t,B} \ot M^t) \circ \alpha(\chi^-_M\ot\chi^-_M)\circ \varphi_{M,M}	\\
&=\theta_{B,M}	\circ	\alpha(\nabla_B\ot\nabla_M) \circ \alpha(B\ot \phi_{M^t,B} \ot M^t) \circ \varphi_{B\ot M,B\ot M}\circ (\alpha(\chi^-_M)\ot\alpha(\chi^-_M))
\intertext{where the second equation follows from the left comodule algebra property and the last equality holds because of the the naturality of $\varphi$. On the other hand, we have}
&\iota \circ \nabla_{\ab \ct_B M} \circ (G_M \ot G_M)	\\
&=\nabla_{\ab \ot M} \circ (\iota\ot\iota) \circ (G_M \ot G_M)	\\
\mbox{by def. G} &=(\nabla_{\ab}\ot\nabla_M) \circ (\ab \ot \phi_{M,\ab} \ot M) \circ (\theta_{B,M}\ot\theta_{B,M}) \circ (\alpha(\chi^-_M) \ot\alpha(\chi^-_M))	\\
\by{defnablaam} &=(\alpha(\nabla_B)\ot\nabla_M) \circ (\varphi_{M,M}\ot B\ot B) \circ (\ab \ot \phi_{M,\ab} \ot M)\\
&\quad\quad\quad\quad\quad \circ (\theta_{B,M}\ot\theta_{B,M}) \circ (\alpha(\chi^-_M) \ot\alpha(\chi^-_M)).\\
\end{align*}
So we're done if we can show
\begin{align*}
&\theta_{B,M}	\circ	\alpha(\nabla_B\ot\nabla_M) \circ \alpha(B\ot \phi_{M^t,B} \ot M^t) \circ \varphi_{B\ot M,B\ot M} 
\\
&=(\alpha(\nabla_B)\ot\nabla_M) \circ (\varphi_{M,M}\ot B\ot B) \circ (\ab \ot \phi_{M,\ab} \ot M) \circ (\theta_{B,M}\ot\theta_{B,M}),
\end{align*}
which can be shown similar to proving that in \eqref{diaomcomodalg} diagrams (III) and (IV) are commutative. We arrive at the following proposition.
\begin{proposition}\label{propamcongahctm}
Let $\alpha : \BC \ra \LC$ be a tensor equivalence trivializable on $\C$ and satisfying (A). Let $M$ be a $B$-$F$-bicomodule algebra, then 
\[
\am \cong \ab \ct_B M
\]
as $L$-$F$-bicomodule algebras in $\C$.
\end{proposition}
Now let $\beta : \LC \rightarrow \BC$ be an 'inverse' functor of the equivalence $\alpha$. We could repeat the same process with $\beta$. I.e. $\bL$ is a flat $B$-$L$-bicomodule algebra and
\begin{align*}
L \cong \alpha(\bL) \cong \ab \ct_B \bL
\end{align*}
as $L$-comodule algebras in $\C$. Similarly, we can show that
\[
B \cong \beta (\ab) \cong \bL \ct_L \ab
\]
as $B$-bicomodule algebras. The following proposition is due to Schauenburg.
\begin{proposition}[{\cite[Proposition 3.4]{schauenburgbr2}}]
Let $L, B$ be flat Hopf algebras in $\C$, and $A$ a flat $L$-$B$-bicomodule algebra. The following are equivalent:
\begin{itemize}
\item[1.] $A$ is a faithfully flat $L$-$B$-bi-Galois object,
\item[2.] there is a flat $B$-$L$-bicomodule algebra $A\inv$ such that $A \ct_B A\inv \cong L$ as $L$-bicomodule algebras and $A\inv \ct_L A \cong B$ as $B$-bicomodule algebras.
\end{itemize}
\end{proposition}
Therefore, we have proven the following theorem.
\begin{theorem}\label{thmahbigal}
Assume $\alpha : \BC \rightarrow \LC$ is a tensor equivalence trivilizable on $\C$ satisfying (A), or equivalently, a right $\C$-module functor satisfying (A). Then $\ab$ is a faithfully flat $L$-$B$-bi-Galois object.
\end{theorem}
The process of assigning an equivalence $\alpha_A = A \ct_B -$ to an $L$-$B$-bi-Galois object $A$ and the process of obtaining a bi-Galois object $\ab$ from an equivalence $\alpha : \BC \ra \LC$ as described above, are obviously mutually inverse. Moreover, this correspondence is compatible with the multiplication of bi-Galois objects and the composition of functors. Indeed, let $B, L, F$ be flat Hopf algebras in $\C$ and suppose $\alpha : \BC \ra \LC$ and $\alpha' : \LC \ra \FC$ are tensor equivalences, trivializable on $\C$ and satisfying (A). Then 
\[
\alpha' (\ab) = \alpha'(L) \ct_L \ab
\]
as $F$-$B$-bicomodule algebras, by Proposition \ref{propamcongahctm}. Hence, we have a group isomorphism between the group of faithfully flat $B$-bi-Galois objects and the group of isomorphism classes of autoequivalences of $\BC$ trivializable on $\C$ and satisfying $(\mathbb{A})$. Let's denote the latter by $\autaf{\BC,\C}$.
\begin{proposition}\label{propbigalisauta}
Let $B$ be a flat Hopf algebra in $\C$, then
\[
\mathrm{BiGal}(B) \cong \Aut_{\scriptscriptstyle \mathbb{(A)}}(\BC,\C)
\]
\end{proposition}

In the following section, we will apply this new result to the specific case where $\C$ is given by the category of left $H$-modules where $H$ is a finite dimensional quasitriangular Hopf algebra.


\section{The Brauer group of a finite quasitriangular Hopf algebra}

Let $(H,R)$ be a quasitriangular Hopf algebra over a field $k$. Dual to the construction in \cite{zhang}, there exists an exact sequence of groups
\[
1 \lra \Br(k) \lra \BM(k,H,R) \overset{\ppi}\lra \gqc{\tr}
\]
Here $Br(k)$ is the (classical) Brauer group of the field $k$, $BM(k,H,R) = Br(\HM)$ is the Brauer group of $H$-module algebras (or equivalently, the Brauer group of the braided monoidal category $\HM$) and $\gqc{\tr}$ is the subgroup of (isomorphism classes) of  quantum commutative $\tr$-bi-Galois objects. In this section, we will give a categorical interpretation to this sequence.

As in Section \ref{secprelim}, we will denote the braidings of the categories $\HM$ and $\yd$ by $\psi$ and $\phi$, respectively.

\subsection{A new characterization for $\mathbf{\gqc{\tr}}$}

\begin{lemma}\label{lemmaydistrcom}
Any left-left Yetter-Drinfeld module $M$ has the structure of a cocommutative $\tr$-bicomodule in the category $\HM$. Conversely, any cocommutative $\tr$-bicomodule is a left-left Yetter-Drinfeld module.

We obtain an equivalence of braided monoidal categories $\yd \ra \trcom$.
\end{lemma}
\begin{proof}[Sketch of proof.]
Given $(M,\cdot,\lambda)\in\yd$, then $M \in {}^{\tr}(\HM)^{\tr}$ via
\begin{align*}
&\chi^- (m) \overset{not.}{=} \Blm{m} = m\hm{-1} S(R^2) \ot R^1 \cdot m\hm{0}	\\
&\chi^+ (m) \overset{not.}{=} \Brm{m} = R^2 \cdot m\hm{0} \ot R^1m\hm{-1}
\intertext{for $m \in M$. It is easy to see}
&\sigma_{\tr,M}\circ\chi^-(m) = \sigma_{\tr,M} ( m\hm{-1} S(P^2) \ot P^1 \cdot m\hm{0})	\\
&\phantom{\sigma_{\tr,M}\circ\chi^-(m) }= r^2R^1 P^1 \cdot m\hm{0} \ot r^1m\hm{-1} S(P^2)R^2	\\
&\phantom{\sigma_{\tr,M}\circ\chi^-(m) }= r^2 \cdot m\hm{0} \ot r^1m\hm{-1} = \chi^+(m)
\intertext{Conversely, given a cocommutative $\tr$-bicomodule $(N,\cdot,\chi^-,\chi^+)$, then $N$ becomes a left-left Yetter-Drinfeld module via}
&\lambda(n)  = n\blm{-1} R^2 \ot R^1\cdot n\blm{0}
\intertext{or}
&\lambda(n)= SR^1 n\brm{1} \ot R^2\cdot n\brm{0}
\end{align*}
for $n\in N$, using the cocommutativity. For a complete proof we refer to \cite[Section 2]{ZZ}.
\\
Finally, we can transfer the braiding of $\yd$ to $\trcom$ such that the equivalence becomes one of braided monoidal categories. That is, if $M, N \in \trcom$. The braiding, again denoted by $\phi$, is then defined by
\[
\phi_{M,N}(m \ot n) = m\blm{-1}R^2 \cdot n \ot R^1 \cdot m\blm{0}
\]
for $m\in M$ and $n \in N$.
\end{proof}

In particular, we see that $X \in {}^{\tr}(\HM)$ has a trivial $\tr$-comodule structure if and only if as a $\yd$-module $X$ is obtained using $\lambda_1$ (i.e. $X \in \mr \subset \yd$). In this case $\phi_{X,-} = \psi_{X,-}$. Furthermore, any (braided) monoidal autoequivalence $\alpha : \yd \ra \yd$ can be seen as a (braided) tensor equivalence $\alpha : \trcom \ra \trcom$ and conversely. 

Following \cite{ZZ}, we will call a braided bi-Galois object $A$ \emph{quantum commutative}\index{Galois object!bi-Galois object!quantum commutative} if $A$ is a cocommutative bi-Galois object which is commutative as an algebra in the category of left-left Yetter Drinfeld modules, that is
\begin{align}\label{eqqc}
ab = (a\hm{-1}\cdot b)a\hm{0}
\end{align}
for all $a, b \in A$. We will denote the group of quantum commutative $\tr$-bi-Galois objects by $\gqc{\tr}$. Clearly, $\gqc{\tr}$ is a subgroup of $\bg{\tr}$. By Proposition \ref{propbigalisauta}, we already know
\begin{align}\label{eqbigalisautaf}
\mathrm{BiGal}(\tr) \cong \autaf{{}^{\tr}(\HM)}{\HM} = \autaf{\yd}{\HM}
\end{align}

The following is due to Zhang and Zhu.

\begin{proposition}[{\cite[Theorem 3.6]{ZZ}}]\label{propqcisbraided}
Let $(H,R)$ be a finite dimensional quasi-triangular Hopf algebra. Suppose $A$ is an $\tr$-bi-Galois object. $A$ is quantum commutative if and only if the functor $A \ct_{\tr} -$ is a braided (monoidal) autoequivalence of the category ${}^{\tr}(\HM)$ (or $\yd$ by Lemma \ref{lemmaydistrcom}).
\end{proposition}
Combining this proposition with our result \eqref{eqbigalisautaf}, we see that the group of quantum commutative bi-Galois objects corresponds precisely to the group of braided monoidal autoequivalences which are trivializable on $\mr$ and satisfying (A), say $\autbra$. Indeed, if $A$ is quantum commutative $\tr$-bi-Galois object, then, as we've discussed in the previous section, $\alpha_A = A \ct_{\tr} -$ is a braided autoequivalence of $\yd$ (use the equivalence in Lemma \ref{lemmaydistrcom}), trivializable on $\mr$ and satisfying diagram (A). 

Conversely, let $\alpha \in \autbra$. In particular, $\alpha \in \autaf{\yd,\mr}$. By Theorem \ref{thmahbigal}, $\alpha(\tr)$ is a faithfully flat $\tr$-bi-Galois object and $\alpha \cong \alpha(\tr)\ct_{\tr} -$. But then by Proposition \ref{propqcisbraided}, $\alpha(\tr)\ct_{\tr} -$ being a braided autoequivalence implies that $\alpha(\tr)$ is quantum commutative.

Next, we show that in this specific case, the condition (A) can be dropped. Namely, any braided monoidal autoequivalence $\alpha : \yd \ra \yd$, or ${}^{\tr}(\HM) \ra {}^{\tr}(\HM)$, which is trivializable on $\HM$ will automatically satisfy diagram (A). Indeed, if $\alpha$ is braided, then
\begin{diagram}
\alpha(U) \ot \alpha(V)	&	\rTo^{\varphi_{U,V}}	&	\alpha(U\ot V)	\\
\dTo^{\phi_{\alpha(U)\ot \alpha(V)}}			&							&	\dTo_{\alpha(\phi_{U,V})}	\\
\alpha(V) \ot \alpha(U)	&	\rTo^{\varphi_{V,U}}	&	\alpha(V\ot U)
\end{diagram}

for any $U,V \in {}^{\tr}(\HM) = \yd$. As observed above, if $X \in {}^{\tr}(\HM)$ has a trivial $\tr$-comodule structure, then $\phi_{X,-} = \psi_{X,-}$. Moreover, $\alpha$ being trivializable implies that $\alpha(X) \cong X$ in ${}^{\tr}(\HM) = \yd$ such that $\alpha(X)$ will also have a trivial $\tr$-comodule structure, hence also $\phi_{\alpha(X),-} = \psi_{\alpha(X),-}$. Thus for $U = X^t$ in the above diagram, we get

\begin{diagram}
\alpha(X^t) \ot \alpha(V)	&	\rTo^{\varphi_{X^t,V}}	&	\alpha(X^t\ot V)	\\
\dTo^{\psi_{\alpha(X^t)\ot \alpha(V)}}			&							&	\dTo_{\alpha(\psi_{X^t,V})}	\\
\alpha(V) \ot \alpha(X^t)	&	\rTo^{\varphi_{V,X^t}}	&	\alpha(V\ot X^t)
\end{diagram}

Thus $\alpha$ satisfies diagram (A). Let's denote the group of isomorphism classes of braided monoidal autoequivalences of $\yd$ trivializable on $\mr$ by $\autbr$. We have obtained the following proposition.

\begin{proposition}\label{qcbigalisequiv} 
Let $(H,R)$ be a finite dimensional quasitriangular Hopf algebra. The group of quantum commutative $\tr$-bi-Galois objects is isomorphic to the group of isomorphism classes of braided (monoidal) autoequivalences of $\yd$ trivializable on $\mr$, that is
\[
\gqc{\tr} \longrightarrow \autbr, \ \  A\mapsto A\Box_{\tr}-
\]
is a group isomorphism with inverse given by $\alpha\mapsto \alpha(\tr)$.
\end{proposition}

\subsection{A new characterization for $\mathbf{BM(k,H,R)}$}

The next goal is to give a new characterization for the Brauer group $BM(k,H,R)$ of $H$-module algebras. First remark that a $(\C,\D)$-bimodule category is the same as a left $\C \boxtimes \D\op$-module category, where $\boxtimes$ denotes the Deligne tensor product of abelian categories (cf. \cite{deligne}). Recall from \cite{etingofostrik} the following definition of an exact module category.
\begin{definition}\label{defexactmodcat}
Let $\C$ be a tensor category. A $\C$-module category $\M$ is said to be \emph{exact}\index{module category!exact} if for any projective object $X \in \C$ and every object $M \in \M$ the object $X \ast M$ is projective in $\M$.
\end{definition}
\begin{definition}[\cite{eno}]\label{definvmodcat}
An exact $\C$-bimodule category $\M$ is said to be \emph{invertible}\index{module category!bimodule category!invertible} if there exists an exact $\C$-bimodule category $\N$ such that
\[
\M \boxtimes_{\C} \N \simeq \N \boxtimes_{\C} \M \simeq \C
\]
where $\C$ is viewed as a $\C$-bimodule category via the regular left and right actions of $\C$.

The group of equivalence classes of invertible $\C$-bimodule categories is called the \emph{Brauer-Picard group of $\C$}\index{Brauer-Picard group} and is denoted BrPic$(\C)$.
\end{definition}

If $\C$ is also braided, we can turn any left $\C$-module category into a $\C$-bimodule category, the right $\C$-action is defined as follows: $M \ast X =: X \ast M$ via the braiding for all $X\in\C$ and $M\in\M$. A $\C$-bimodule category is said to be \emph{one-sided}\index{module category!bimodule category!one-sided} if it is equivalent to a bicomodule category with right $\C$-action induced from the left, as just described. Therefore, when $\C$ is braided, the group BrPic$(\C)$ contains a subgroup Pic$(\C)$ consisting of equivalence classes of one-sided invertible $\C$-bimodule categories. Pic$(\C)$ is called the \emph{Picard group of $\C$}\index{Picard group}.

If $A$ is an algebra in $\C$, the category of right $A$-modules in $\C$ is naturally a left $\C$-module category via
\[
\C \times \C_A \ra \C_A,\ (X,M) \mapsto X\ot M
\]
Here the object $X\ot M$ has the structure of a right $A$-module in $\C$ via $X \ot \mu^+$, where $\mu^+ : M \ot A \ra M$ denotes the right $A$-action on $M$.

We quote
\begin{proposition}[{\cite[Proposition 3.4]{davydovnikshych}}]\label{propcacb}
Let $\C$ be a braided tensor category and let $A$ and $B$ be exact algebras in $\C$. Then
\begin{align}\label{eqcacb}
\C_A \btc \C_B \simeq \C_{A\ot B}
\end{align}
\end{proposition}

As ${}_A\C$ considered as a right $\C$-module category is equivalent to $\C_{\oA}$, we obtain
\begin{align}\label{eqacbc}
{}_A\C \btc {}_B\C \simeq \C_{\oA} \btc \C_{\oB} \simeq \C_{\oA\ot\oB} \cong \C_{\ol{B\ot A}} \simeq {}_{B\ot A}\C
\end{align}
\mbox{}\\[1ex]
Let $(H,R)$ be a finite dimensional quasitriangular Hopf algebra and let $\C$ be the braided monoidal category $\HM$. We can relate the Picard group of $\HM$ to the Brauer group of $\HM$. It is claimed in \cite{davydovnikshych} that the Picard group of any braided tensor category $\C$ is isomorphic to the group of Morita equivalence classes of exact Azumaya algebras (where an algebra $A$ is said to be \emph{exact} if the category $\C_A$ is exact). We show that, for $\C = \HM$, any Azumaya algebra is exact. Accordingly, the Picard group of $\HM$ will be isomorphic to the Brauer group of $\HM$. Let us give a complete proof in the following proposition.

\begin{proposition}\label{picisbr}
The Picard group of $\C$ is isomorphic to the Brauer group of $\HM$.
\[
\Pic(\HM) \cong \BM(k,H,R)
\]
\end{proposition}
\begin{proof}
Assume $A$ is $H$-Azumaya. In particular, $A$ is an algebra in $\C$. Moreover 
\[
\C_A \simeq {}_{\oA}\C = {}_{\oA} (\HM ) =  {}_{\oA \# H}\M
\]
where $\oA$ is the opposite algebra. $\oA \# H$ is a right $H$-comodule algebra with right coaction $\rho(a\# h) = (a \# h_1) \ot h_2$ for $a \in \oA$, $h\in H$. If we can show that $\oA \# H$ is $H$-simple, then $\oA \# H$ is exact by \cite[Proposition 1.20(i)]{andruskiewitschmombelli}. To prove that $\oA \# H$ is $H$-simple, it is sufficient to show that $\oA$ is $H$-simple. Indeed, let $J$ be an $H$-ideal of $\oA \# H$. One can check that $J$ is an $H$-Hopf module. By the Fundamental Theorem of Hopf modules, we obtain $J \cong I \ot H$ as $H$-Hopf modules, where $I = J^{coH}$. $I$ will then be an $H$-ideal of $\oA$. If $\oA$ is shown to be $H$-simple, $I$ must be trivial, implying that either $J \cong H$ or $J \cong \oA \# H$. In the first case however, $J$ will not be an $H$-ideal of $\oA \# H$. Thus $\oA \# H$ will not contain a non-trivial $H$-ideal if $\oA$ is $H$-simple. So let us show that $\oA$ is $H$-simple. Let $J$ be a non-trivial $H$-ideal of $\oA$, in particular $J$ is an $H$-submodule ideal of $\oA$. Consider $A\ot \oA$ which has the braided product
\begin{align}\label{eqaeproduct}
&(a \ot \ol{b} )(c \ot \ol{d} ) = a (R^2\cdot c) \ot  \ol{(r^2\cdot d)(r^1R^1\cdot b)}	
\end{align}
for $a,b,c,d \in A$. Consider the subset $A \ot J$ which is now easily seen to be a non-trivial ideal of $A \ot \oA$. The latter is an Azumaya algebra ($A$ is, hence so are $\oA$ and $A\ot\oA$). Since $k$ is a field, $A \ot \oA$ is simple. Contradiction. Thus $\oA$ is $H$-simple. Hence, so is $\oA \# H$, and therefore $\oA \# H$ is exact. Whence ${}_{\oA \# H}\M = \C_A$ is an exact module category.

By Proposition \ref{propcacb} and the definition of an Azumaya algebra in $\C$, we have
\begin{align*}
&\C_{\oA} \btc \C_A \simeq \C_{\oA\ot A} \simeq \C
\intertext{Similarly by using \eqref{eqacbc}, we get}
&\C_A \btc \C_{\oA} \simeq {}_{\oA}\C \btc {}_A\C \simeq {}_{A\ot\oA}\C \simeq \C
\end{align*}
Hence, $\C_A$ is an exact invertible (one-sided) module category over $\C$.

Conversely, let $\M$ be an exact invertible (one-sided) module category. By \cite[Proposition 4.2]{eno} we have
\[
\M \btc \M\op \simeq \M\op \btc \M \simeq \C
\]
By \cite[Theorem 1.14]{andruskiewitschmombelli} there exists an (exact) algebra $A$ in $\HM$ such that $\M$ is equivalent to $\C_A$. Then $\M\op \simeq \C_{\oA}$. Again, by using \eqref{eqcacb} and \eqref{eqacbc}, we see
\begin{align*}
&\C \simeq \M\op \btc \M \simeq \C_{\oA} \btc \C_A \simeq \C_{\oA\ot A}	\\
&\C \simeq \M \btc \M\op \simeq \C_A \btc \C_{\oA} \simeq {}_{\oA}\C \btc {}_A\C \simeq {}_{A\ot\oA}\C
\end{align*}
Thus by definition, the algebra $A$ is Azumaya.

Finally, the correspondence is one of groups because of Proposition \ref{propcacb}.
\end{proof}

\subsection{A new characterization for $\mathbf{\ppi : \BM(k,H,R) \ra \gqc{\tr}}$ }
As we have found categorical characterizations for both groups $\BM(k,H,R)$ and $\gqc{\tr}$, it should be necessary to obtain an interpretation for the morphism $\ppi : \BM(k,H,R) \ra \gqc{\tr}$.  Let $H$ be a finite dimensional Hopf algebra. 
It is obvious that $\HM = \MHs$ as braided monoidal categories. One can further identify the two Yetter-Drinfeld module categories naturally:
\begin{equation}\label{lemmaydident}
 \yd \cong \dy
\end{equation}

Let us recall how the morphism $\ppi : \BM(k,H,R) \ra \gqc{\tr}$ is defined.  Suppose $A$ is an Azumaya algebra in the braided monoidal category $\HM = \MHs$, that is $[A] \in \BM(k,H,R) = \mathrm{BC}(k,H^*,\R)$, where $\R=R^*$. In \cite[Corollary 4.2]{zhang}) it was shown that any element of $\mathrm{BC}(k,H^*,\R)$ can be represented by an Azumaya algebra that is a smash product. Any smash product is a Galois extension of its coinvariants. Thus, any element of $\mathrm{BC}(k,H^*,\R)$ can be represented by an Azumaya algebra that is an $H^*$-Galois extension of its coinvariants. As a result, we're allowed to assume that our Azumaya algebra $A \in \HM=\MHs$ is $H^*$-Galois over its coinvariant subalgebra $A_0 = A^{coH^*}$. Observe that
\begin{align}\label{eqazero}
A_0 = A^{co{H^*}} = {}_HA = \{ a\in A \st h \cdot a = \epsilon(h)a, \forall h \in H \}
\end{align}
Now $\pi(A)$ is defined as the centralizer subalgebra $C_A(A_0)$ of $A_0$ in $A$. Then $\pi(A) \in \dy$ where $\pi(A)$ is an $H^*$-subcomodule of $A$ and the right $H^*$-action is the Miyashita-Ulbrich-action (or MU-action), given by
\[
c \lhu h^* = x_i(h^*) c y_i(h^*)
\]
for $c\in\pi(A)$ and $h^*\in H^*$, where $x_i(h^*)\ot y_i(h^*) = can\inv (1 \ot h^*)$. By the identification (\ref{lemmaydident}), we get that $\pi(A) \in \yd$ is an $H$-submodule of $A$ and the left $H$-coaction is dual to the MU-action. As a left-left Yetter-Drinfeld module, $\pi(A)$ can be seen as an $_RH$-bicomodule. Finally, it is shown that $\pi(A)$ is a bi-Galois object and the morphism $\ppi : \BM(k,H,R) \ra \gqc{\tr}$ sends a class $[A]$ to $[\pi(A)]$.
\\\\
For the remainder of the paper, we will graphically denote the braiding of $\HM$ (and its inverse) by
\begin{align}
\psi_{M,N}
=
\gbeg{2}{3}
\got{1}{M}\got{1}{N}\gnl
\gbr\gnl
\gob{1}{N}\gob{1}{M}
\gend
&\quad , \quad
\psi_{M,N}\inv
=
\gbeg{2}{3}
\got{1}{N}\got{1}{M}\gnl
\gibr\gnl
\gob{1}{M}\gob{1}{N}
\gend
\label{notbraiding1}
\intertext{for $M,N \in \HM$, while the braiding of $\yd$ (and its inverse) is denoted respectively}
\phi_{X,Y}
=
\gbeg{2}{3}
\got{1}{X}\got{1}{Y}\gnl
\gbrc\gnl
\gob{1}{Y}\gob{1}{X}
\gend
&\quad , \quad
\phi_{X,Y}\inv
=
\gbeg{2}{3}
\got{1}{Y}\got{1}{X}\gnl
\gibrc\gnl
\gob{1}{X}\gob{1}{Y}
\gend
\label{notbraiding2}
\end{align}
for $X,Y \in \yd$.

We need the following lemma in the sequel. 
\begin{lemma}Let $(H,R)$ be a quasitriangular Hopf algebra. Then
\begin{align}\label{eqqybe1}
u^1 p^1 U^1 \ot u^2 r^1 R^1 \ot p^2 r^2 \ot U^2 R^2
=
p^1 U^1 u^1 \ot r^1 R^1 u^2 \ot r^2 p^2 \ot R^2 U^2
\end{align}
\end{lemma}
\begin{proof}A quasitriangular Hopf algebra $(H,R)$ is known to satisfy the \emph{quantum Yang-Baxter equation}, that is
\[
R_{12} R_{13} R_{23} = R_{23} R_{13} R_{12}
\]
where $R_{12} = R^1 \ot R^2 \ot 1_H \in H \ot H \ot H$, etc. Or
\begin{align}\label{qybe}\tag{QYBE}
R^1 P^1 \ot R^2 Q^1 \ot P^2 Q^2 = P^1 R^1 \ot Q^1 R^2 \ot Q^2 P^2.
\end{align}
By \eqref{qybe} we have $R_{12} R_{13} = R_{23} R_{13} R_{12} (R\inv)_{23}$ or
\[
u^1 p^1 \ot u^2 \ot p^2 = p^1 u^1 \ot q^1u^2 S(v^1) \ot q^2 p^2 v^2
\]
Then
\begin{align*}
&u^1 p^1 U^1 \ot u^2 r^1 R^1 \ot p^2 r^2 \ot U^2 R^2	\\
&=p^1 u^1 U^1 \ot q^1 u^2 S(v^1) r^1 R^1 \ot q^2 p^2 v^2 r^2 \ot U^2 R^2	\\
&=p^1 u^1 U^1 \ot q^1 u^2  R^1 \ot q^2 p^2  \ot U^2 R^2	\\
&=p^1 U^1 u^1 \ot q^1 R^1 u^2 \ot q^2 p^2  \ot R^2 U^2
\end{align*}
\end{proof}

In the the following, $A$ is still assumed to be an Azumaya algebra in $\HM$ which is $H^*$-Galois over $A_0$. Let $Z$ be a left-left Yetter-Drinfeld module. We can define an $A^e$-module structure on $A \ot Z$ as follows
\begin{align}\label{defazaemod}
\mu_{A\ot Z}
=
\gbeg{4}{6}
\got{2}{A\ot\oA}\got{2}{A\ot Z} \gnl
\gcl{1}\gbr\gcl{1}\gnl
\gmu\gibrc\gnl
\gcn{2}{1}{2}{2}\gbr\gnl
\gwmuh{3}{2}{5}\gcl{1}\gnl
\gvac{1}\gob{1}{A}\gvac{1}\gob{1}{Z}
\gend
\end{align}
with notation as in \eqref{notbraiding1},\eqref{notbraiding2}, or equivalently
\[
(a\ot \ol{b}) \bullet (c \ot z) = a (R^2\cdot c)(r^2 S\inv (z\hm{-1})R^1 \cdot b) \ot r^1 \cdot z\hm{0}
\]
for $a,b,c \in A$ and $z\in Z$. Indeed
\begin{align*}
&[(a \ot \ob)(c \ot \ol{d})]\bullet (e \ot z) = [a (R^2\cdot c) \ot  \ol{(r^2\cdot d)(r^1R^1\cdot b)}] \cdot (e \ot z)\\
&= a(R^2\cdot c)(P^2\cdot e)(Q^2S\inv(z\hm{-1})P^1\cdot [(r^2\cdot d)(r^1R^1\cdot b)] \ot Q^1\cdot z\hm{0}\\
&=a(R^2\cdot c)(P^2p^2\cdot e)(Q^2 S\inv (z\hm{-1})P^1r^2\cdot d)(q^2 S\inv(z\hm{-2})p^1r^1R^1\cdot b) \ot q^1 Q^1 \cdot z\hm{0}	\hide{QT2,3}\\
&=a(R^2\cdot c)(p^2P^2\cdot e)(Q^2 S\inv (z\hm{-1})r^2P^1\cdot d)(q^2 S\inv (z\hm{-2})r^1p^1R^1\cdot b) \ot q^1 Q^1 \cdot z\hm{0}	\hide{\by{qybe}}\\
&=a(R^2\cdot c)(p^2P^2\cdot e)(Q^2 r^2S\inv (z\hm{-2})P^1\cdot d)(q^2 r^1S\inv (z\hm{-1})p^1R^1\cdot b) \ot q^1 Q^1 \cdot z\hm{0}	\hide{\by{qt4}}
\end{align*}
where the first equation follows from (\ref{eqaeproduct}), the second and the fourth hold by (\ref{defazaemod}) and \eqref{eqqybe1}
On the other hand, we have 
\begin{align*}
&(a \ot \ob ) \bullet [(c \ot \ol{d})\bullet (e \ot z)]	\\
&=(a \ot \ob ) \bullet [c (P^2\cdot e)(U^2 S\inv (z\hm{-1})P^1 \cdot d) \ot U^1 \cdot z\hm{0}]	\dome{defazaemod}\\
&=a(R^2\cdot [c (P^2\cdot e)(U^2 S\inv (z\hm{-1})P^1 \cdot d)])(q^2 S\inv ((U^1 \cdot z\hm{0})\hm{-1})R^1 \cdot b) \\
		&\hspace{7cm}\ot q^1 \cdot ((U^1 \cdot z\hm{0})\hm{0})\dome{defazaemod}\\
&=a(R^2\cdot c)(p^2P^2\cdot e)(u^2U^2 S\inv (z\hm{-1})P^1 \cdot d)(q^2 S\inv ((U^1 \cdot z\hm{0})\hm{-1})u^1p^1R^1 \cdot b) \\
		&\hspace{7cm}\ot q^1 \cdot ((U^1 \cdot z\hm{0})\hm{0})	\hide{QT3}\\
&=a(R^2\cdot c)(p^2P^2\cdot e)(u^2U^2Q^2r^2 S\inv (z\hm{-1})P^1 \cdot d)(q^2 S\inv (U^1 z\hm{-1} S(r^1))u^1p^1R^1 \cdot b) \\
		&\hspace{7cm}\ot q^1 Q^1 \cdot z\hm{0}	\hide{YD,QT2}\\
&=a(R^2\cdot c)(p^2P^2\cdot e)(Q^2r^2 S\inv (z\hm{-2})P^1 \cdot d)(q^2 r^1 S\inv (z\hm{-1})p^1R^1 \cdot b) \ot q^1 Q^1 \cdot z\hm{0}
\end{align*}
for $a,b,c,d,e \in A$ and $z\in Z$.
\begin{lemma}\label{lemmaazaeydmod}
Let $A$ be an \hide{$H^*$-Galois }Azumaya algebra in $\HM$ and let $Z$ be a left-left Yetter-Drinfeld module. $A \ot Z$ is a left-left Yetter-Drinfeld module with structures given by
\begin{equation}\label{defazydmod}
\begin{split}
&h \cdot (a \ot z) = h_1 \cdot a \ot h_2 \cdot z		\\
&\lambda_{A\ot Z} (a \ot z) = {SR^1} z_{-1} \ot {R^2\cdot a} \ot z_0
\end{split}
\end{equation}
for $a \in A$ and $z \in Z$.
\\
Together with the $\Ae$-module structure, $A \ot Z$ becomes an object in ${}_{\Ae}(\yd)$.
\end{lemma}
\begin{proof}
We embed $A$ in $\yd$ by viewing it as $A \in (\mri,\psi)$, that is $A$ is equipped with the $H$-coaction $\lambda_2(a) = SR^1 \ot R^2 \cdot a$ (see \eqref{deflambda1en2}). We use $\lambda_2$ rather than $\lambda_1$. Using $\lambda_1$ we would not necessarily yield that $A\ot Z$ becomes an object in ${}_{\Ae}(\yd)$. That being said, there is a reason to choose $\lambda_2$ over $\lambda_1$ anyhow. The reason will become clear in the proof of Proposition \ref{propaotzapictz}.

The given Yetter-Drinfeld module structure in the lemma now is nothing but the natural Yetter-Drinfeld module structure of the tensor product of the two left-left Yetter-Drinfeld modules $A$ and $Z$. 
\\
$\Ae$ is an $H$-module algebra with multiplication as in \eqref{eqaeproduct}. Note that if we consider $A$ to be in $(\mri,\psi)$, we have to view $\Ae$ in $(\mri,\psi)$ as well. Thus its $H$-comodule structure is given by $\lambda_{\Ae}(a\ot \ob) = SR^1 Sr^1 \ot (R^2 \cdot a \ot \ol{r^2 \cdot b})$. Let us verify that $\Ae$ then is an algebra in $\yd$.\dome{note about braiding?}
\begin{align*}
&\lambda_{\Ae}((a\ot \ob) (c \ot \ol{d})) = \lambda_{\Ae}(a (R^2\cdot c) \ot  \ol{(r^2\cdot d)(r^1R^1\cdot b)})	\\
&=S(P^1)S(p^1) \ot P^2 \cdot (a (R^2\cdot c)) \ot \ol{ p^2\cdot ((r^2\cdot d)(r^1R^1\cdot b))}	\\
&=S(P^1)S(U^1)S(p^1)S(u^1) \ot (P^2 \cdot a)(U^2 R^2\cdot c) \ot \ol{(p^2r^2\cdot d)(u^2r^1R^1\cdot b)}	\hide{QT3}	\\
&=S(P^1)S(u^1)S(U^1)S(p^1) \ot (P^2 \cdot a)(R^2U^2 \cdot c) \ot \ol{(q^2p^2\cdot d)(q^1R^1u^2\cdot b)}	\hide{\by{eqqybe1}}		\\
&=S(P^1)S(u^1)S(U^1)S(p^1) \ot (P^2 \cdot a \ot \ol{u^2\cdot b})(U^2 \cdot c \ot\ol{p^2\cdot d})	\\
&=\lambda_{\Ae}(a\ot \ob) \lambda_{\Ae}(c \ot \ol{d})
\end{align*} 
To observe that $A \ot Z$ is an object in ${}_{\Ae}(\yd)$, observe that 
\begin{align*}
&h \cdot [(a\ot \ol{b}) \bullet (c \ot z)]	\\
&= (h_1 \cdot a\ot \ol{h_2\cdot b}) \bullet  (h_3\cdot c \ot h_4\cdot z)	\\
&= (h_1 \cdot a) (R^2h_3 \cdot c)(r^2 S\inv ((h_4 \cdot z)\hm{-1})R^1h_2 \cdot b) \ot r^1 \cdot (h_4 \cdot z)\hm{0} \dome{defazaemod}	\\
&= (h_1 \cdot a) (R^2h_3 \cdot c)(r^2 S\inv (h_4 z\hm{-1}S(h_6))R^1h_2 \cdot b) \ot r^1 h_5 \cdot z\hm{0}	\dome{YD}\\
&= (h_1 \cdot a) (h_2R^2 \cdot c)(r^2 h_6S\inv(z\hm{-1}) S\inv(h_4)h_3R^1 \cdot b) \ot r^1 h_5 \cdot z\hm{0}	\hide{\by{qt4}}\\
&= (h_1 \cdot a) (h_2R^2 \cdot c)(r^2 h_4S\inv(z\hm{-1}) R^1 \cdot b) \ot r^1 h_3 \cdot z\hm{0}	\\
&= (h_1 \cdot a) (h_2R^2 \cdot c)(h_3 r^2 S\inv(z\hm{-1}) R^1 \cdot b) \ot h_4 r^1  \cdot z\hm{0}\hide{\by{qt4}}\\
&= h \cdot [a (R^2\cdot c)(r^2 S\inv (z\hm{-1})R^1 \cdot b) \ot r^1 \cdot z\hm{0}]	\\
&= h \cdot [(a\ot \ol{b}) \bullet (c \ot z)]\dome{defazaemod}
\intertext{which means that the module structure in \eqref{defazaemod} is $H$-linear. To show that it's $H$-colinear as well, we compute}
&\lambda_{A\ot Z}((a\ot \ol{b}) \bullet (c \ot z))	\\
&=\lambda_{A\ot Z}(a (R^2\cdot c)(r^2 S\inv (z\hm{-1})R^1 \cdot b) \ot r^1 \cdot z\hm{0})	\dome{defazaemod}\\
&=S(P^1) (r^1 \cdot z\hm{0})\hm{-1} \ot P^2 \cdot [a (R^2\cdot c)(r^2 S\inv (z\hm{-1})R^1 \cdot b)] \ot (r^1 \cdot z\hm{0})\hm{0}	\\
&=S(P^1)S(U^1)S(V^1) r^1 z\hm{-1} S(q^1) \ot (P^2 \cdot a)(U^2R^2\cdot c)(V^2r^2p^2q^2S\inv (z\hm{-2})R^1 \cdot b) \\
		&\hspace{5cm} \ot p^1 \cdot z\hm{0}	\hide{\by{qt4}}\dome{YD,QT2,3,4}\\
&=S(P^1)S(U^1)  S(q^1) z\hm{-2}  \ot (P^2 \cdot a)(U^2R^2\cdot c)(p^2S\inv (z\hm{-1})q^2R^1 \cdot b)  \ot p^1 \cdot z\hm{0}	\hide{\by{qt4}}\\
&=S(P^1)S(q^1)  S(U^1) z\hm{-2}  \ot (P^2 \cdot a)(R^2U^2\cdot c)(p^2S\inv (z\hm{-1})R^2q^1 \cdot b)  \ot p^1 \cdot z\hm{0}	\hide{\by{qybe}}\\
&=S(P^1)S(q^1)  S(U^1) z\hm{-1}  \ot (P^2 \cdot a \ot q^1 \cdot b) \bullet (U^2\cdot c  \ot p^1 \cdot z\hm{0})\dome{defazaemod}\\
&=(a\ot \ob)\hm{-1} (c \ot z)\hm{-1} \ot (a\ot \ob)\hm{0} \bullet (c \ot z)\hm{0}
\end{align*}
\end{proof}

Let $A$ be an $H^*$-Galois Azumaya algebra in $\HM$ and let $Z$ be a left-left Yetter-Drinfeld module. The previous lemma allows us to consider $(A \ot Z)^A$, where $(-)^A$ is from the equivalence pair
\[
A \overset{\sim}{\ot} - : \yd \rightleftarrows {}_{\Ae}(\yd) : (-)^A
\]
(similarly) as in \cite[Proposition 2.6]{cvozyetter2}. Then
\begin{equation}\label{endofunctor}
(A \ot Z)^A = \{ \sum c_i \ot z_i \st \sum (a \ot 1)\bullet (c_i \ot z_i) = \sum (1 \ot \oa) \bullet (c_i \ot z_i), \forall a \in A \}
\end{equation}
is a Yetter-Drinfeld submodule of $A \ot Z$ with the same Yetter-Drinfeld module structures as in \eqref{defazydmod}. 

By Lemma \ref{lemmaydistrcom}, we can view $Z$ as a left $\tr$-comodule in $\HM$, thus we can consider $\pi(A) \ct_{\tr} Z$. We will relate $(A \ot Z)^A$ and $\pi(A) \ct_{\tr} Z$, but first we need an equivalent characterization for the cotensor product of two $\tr$-bicomodules in $\HM$ (which is dual to \cite[Lemma 2.9]{zhang}).

\begin{lemma}
If X and Y are two YD modules, then
\[
X \ct_{_RH} Y = \{ \sum {x_i} \ot {y_i} \in X \ot Y \st {x_i}\hm{-1}R^2 \ot {x_i}\hm{0} \ot R^1 \cdot {y_i} = S(R^1) {y_i}\hm{-1} \ot R^2\cdot {x_i} \ot {y_i}\hm{0} \}
\]
\end{lemma}
\begin{proof}
Given $\sum {x_i} \ot {y_i} \in X \ct_{_RH} Y$, using the identification in Lemma \ref{lemmaydistrcom}, we get
\begin{align*}
&\sum {x_i}\hm{-1}R^2 \ot {x_i}\hm{0} \ot R^1 \cdot {y_i}	\\
&=\sum S(r^1) {x_i}\brm{1} R^2 \ot r^2\cdot {x_i}\brm{0} \ot R^1 \cdot {y_i}	\\
&=\sum S(r^1) {y_i}\blm{-1} R^2 \ot r^2\cdot {x_i} \ot R^1 \cdot {y_i}\blm{0}	\\
&=\sum S(r^1) {y_i}\hm{-1}S(P^2) R^2 \ot r^2\cdot {x_i} \ot R^1P^1 \cdot {y_i}\hm{0}\\
&=\sum S(r^1) {y_i}\hm{-1} \ot r^2\cdot {x_i} \ot  \cdot {y_i}\hm{0}\\
\intertext{Conversely, suppose $\sum {x_i} \ot {y_i}$ belongs to the set on the right hand side, then}
&\sum \chi^+({x_i}) \ot {y_i} = \sum R^2 \cdot {x_i}\hm{0} \ot R^1{x_i}\hm{-1} \ot {y_i}	\\
&=\sum R^2 \cdot {x_i}\hm{0} \ot R^1 {x_i}\hm{-1} u^2 S(v^2) \ot v^1 u^1 \cdot {y_i} \\
&=\sum R^2u^2 \cdot {x_i} \ot R^1 S(u^1) {y_i}\hm{-1} S(v^2) \ot v^1 \cdot {y_i}\hm{0} \\
&=\sum {x_i} \ot {y_i}\hm{-1} S(v^2) \ot v^1 \cdot {y_i}\hm{0} = \sum {x_i} \ot \chi^-({y_i})
\end{align*}
\end{proof}
This implies that we can define a coaction on $X \ct_{_RH} Y$
\[
\lambda (\sum {x_i} \ot {y_i}) = \sum {x_i}\hm{-1}R^2 \ot {x_i}\hm{0} \ot R^1 \cdot {y_i} = \sum S(R^1) {y_i}\hm{-1} \ot R^2\cdot {x_i} \ot {y_i}\hm{0}
\]
Together with the diagonal H-action
\[
h \cdot \sum {x_i} \ot {y_i} = \sum h_1 \cdot x_i \ot h_2 \cdot y_i	
\]
this will define a left-left Yetter-Drinfeld module structure on $X \ct_{_RH} Y$. Indeed
\begin{align*}
&\lambda(h \cdot (\sum {x_i} \ot {y_i})) = \sum \lambda (h_1 \cdot {x_i} \ot h_2 \cdot {y_i})	\\
&=\sum S(R^1) (h_2 \cdot {y_i})\hm{-1} \ot R^2h_1\cdot {x_i} \ot (h_2 \cdot {y_i})\hm{0}\\
&=\sum S(R^1) h_2 {y_i}\hm{-1} S(h_4) \ot R^2 h_1 \cdot {x_i} \ot h_3 \cdot {y_i}\hm{0}\\
&=\sum h_1 S(R^1) {y_i}\hm{-1} S(h_4) \ot h_1 R^2 \cdot {x_i} \ot h_3 \cdot {y_i}\hm{0}\\
&=\sum h_1 (S(R^1) {y_i}\hm{-1}) S(h_3) \ot h_2 \cdot (R^2 \cdot {x_i} \ot {y_i}\hm{0})
\end{align*}
The corresponding $_RH$-bicomodule structure is the natural $_RH$-structure of the \mbox{cotensor} product. E.g.
\begin{align*}
&\sum \chi^-_{X \ct_{_RH} Y}({x_i} \ot {y_i}) = \sum ({x_i} \ot {y_i})\hm{-1} S(R^2) \ot R^1 \cdot ({x_i} \ot {y_i})\hm{0}	\\
&=\sum {x_i}\hm{-1}r^2S(R^2) \ot R^1 \cdot ({x_i}\hm{0} \ot r^1 \cdot {y_i})	\\
&=\sum {x_i}\hm{-1}r^2S(P^2)S(R^2) \ot R^1 \cdot {x_i}\hm{0} \ot P^1r^1 \cdot {y_i}	\\
&=\sum {x_i}\hm{-1}S(R^2) \ot R^1 \cdot {x_i}\hm{0} \ot {y_i}	= \sum \chi^-_X({x_i}) \ot {y_i}
\end{align*}
and similarly $\chi^+_{X \ct_{_RH} Y} = X \ot \chi^+_Y$. Thus we could have deduced the Yetter-Drinfeld module structure from the natural $\tr$-bicomodule structure as well. 

Particularly for $\pi(A) \ct_{_RH} Z $, we get
\begin{align}\label{defpiaydmod}
\lambda (\sum c_i \ot z_i) = \sum {c_i}_{-1}R^2 \ot {c_i}_{0} \ot R^1 \cdot z_i = \sum SR^1 {z_i}_{-1} \ot R^2\cdot c \ot {z_i}_0
\end{align}
for $\sum c_i \ot z_i \in \pi(A) \ct_{_RH} Z $.
\\
The following statement is inspired by \cite[Lemma 4.5]{zhang}.

\begin{theorem}\label{propaotzapictz}
Let $A$ be an $H^*$-Galois Azumaya algebra in $\HM$ and let $Z$ be a left-left Yetter-Drinfeld module structure. Then
\[
(A \ot Z)^A = \pi(A) \ct_{\tr} Z 
\]
as left-left Yetter-Drinfeld modules.
\end{theorem}
\begin{proof}
Let $\sum c_i \ot z_i \in \pi(A) \ct_{\tr} Z$ and denote $c \ot z = \sum c_i \ot z_i$, then
\begin{equation*}
\begin{split}
\chi^+ (c) \ot z &= c \ot \chi^-(z)		\\
R^2 \cdot c\hm{0} \ot R^1 c\hm{-1} \ot z &= c \ot z\hm{-1} S(R^2) \ot R^1 \cdot z\hm{0}
\end{split}\tag{$\dagger$}
\end{equation*}
Observe that
\begin{align*}
&(1 \ot \oa) \bullet (c \ot z) = (R^2\cdot c)(r^2 S\inv (z\hm{-1})R^1 \cdot a) \ot r^1 \cdot z\hm{0}	\\
&=(R^2 r^2 \cdot c\hm{0})(S\inv (c\hm{-1}) S\inv (r^1) R^1 \cdot a) \ot z	\tag*{by ($\dagger$)}\\
&=c\hm{0}(S\inv (c\hm{-1}) \cdot a) \ot z	\\
&=(c\hm{-1}S\inv (c\hm{-2}) \cdot a)c\hm{0}	\ot z	\by{eqqc}\\
&=ac \ot z = (a \ot 1) \bullet (c \ot z)
\end{align*}
whence $\pi(A) \ct_{\tr} Z \subset (A \ot Z)^A$. Conversely let $c \ot z = \sum c_i \ot z_i \in (A \ot Z)^A$. Then
\begin{align*}
ac \ot z = (R^2\cdot c)(r^2 S\inv (z\hm{-1})R^1 \cdot a) \ot r^1 \cdot z\hm{0}	\tag*{($\ddagger$)}
\end{align*}
for all $a \in A$. But then by \eqref{eqazero}, we obtain $ac \ot z = ca \ot z$ for all $a \in A_0$. Thus $c \ot z \in \pi(A) \ot Z$. As $A$ is assumed to be $H^*$-Galois, we know that $\pi(A) \in \dy$ where $c \lhu h^* = x_i(h^*) c y_i(h^*)$. With the identification $\dy = \yd$ from Lemma \ref{lemmaydident} in mind, note that
\begin{align}
&x_i(h^*)y_i(h^*)\hrm{0} \ot y_i(h^*)\hrm{1} = 1_A \ot h^*	\tag*{by Lemma \ref{lemmagammaproperties}}	\nonumber
\intertext{or equivalently}
&x_i(h^*)\hrm{0}y_i(h^*) \ot x_i(h^*)\hrm{1} = 1_A \ot S(h^*)					\nonumber
\intertext{hence}
&x_i(h^*)\hrm{0}y_i(h^*) \ \langle x_i(h^*)\hrm{1} , h \rangle = 1_A \langle S(h^*), h \rangle			\nonumber
\intertext{or}
&(h \cdot x_i(h^*))y_i(h^*) = 1_A \langle h^*, S(h) \rangle		\label{eqxiyi}
\end{align}
for $h \in H$ and $h^* \in H^*$. We can now show that $c \ot z \in \pi(A) \ct_{\tr} Z$, that is $c \ot z$ satisfies $(\dagger)$. Let $h^* \in H^*$ and compute
\begin{align*}
&R^2 \cdot c\hm{0} \ot z \ \langle h^*, R^1 c\hm{-1} \rangle		\\
&=R^2 \cdot c\hm{0} \ot z \ \langle h^*_1, R^1 \rangle \langle h^*_2, c\hm{-1} \rangle	\\
&=R^2 \cdot (c \lhu h^*_2) \ot z \ \langle h^*_1, R^1 \rangle		\\
&=R^2 \cdot (x_i(h^*_2)c y_i(h^*_2)) \ot z \ \langle h^*_1, R^1 \rangle	\\
&=R^2 \cdot ((P^2\cdot c)(p^2S\inv (z\hm{-1})P^1 \cdot x_i(h^*_2)) y_i(h^*_2)) \ot p^1 \cdot z \ \langle h^*_1, R^1 \rangle	\tag*{by ($\ddagger$)}\\
&=(R^2P^2\cdot c) \ot p^1 \cdot z \ \langle h^*_1, R^1 \rangle \langle h^*_2, S(p^2S\inv (z\hm{-1})P^1) \rangle	\by{eqxiyi} \\
&=(R^2P^2\cdot c) \ot p^1 \cdot z \ \langle h^*, R^1 S(P^1) z\hm{-1} S(p^2) \rangle	\\
&=c \ot p^1 \cdot z \ \langle h^*,  z\hm{-1} S(p^2) \rangle
\end{align*}
Hence $c \ot z \in \pi(A) \ct_{\tr} Z$.
\\
As $\pi(A)$ is an $H$-submodule of A and both objects have the diagonal $H$-action, the equality is H-linear. The $H$-colinearity is clear from the definition of the $H$-coactions in \eqref{defazydmod} and \eqref{defpiaydmod}. Moreover, this illustrates why in the definition of $\lambda_{A \ot Z}$ we have opted for $\lambda_2$ instead of $\lambda_1$.
\end{proof}

As a consequence of Theorem \ref{propaotzapictz} and Proposition \ref{qcbigalisequiv}, we obtain the following characterization of the $\tr$-bi-Galois object $\pi(A)$ for an $H^*$-Galois Azumaya algebra $A$.

\begin{corollary}\cite[Lemma 4.5]{zhang}\label{piaisaottra}
Let $A$ be an Azumaya algebra in $\HM$ such that it is $H^*$-Galois over $A_0$. Then
\[
\pi(A) \cong (A \ot \tr)^A
\]
as left-left Yetter-Drinfeld module algebras, or equivalently, as $\tr$-bicomodule algebras. Thus they represent the same object in the group of quantum commutative Galois objects $\galqc$.
\end{corollary}

Now let $A$ be an arbitrary  Azumaya algebra in $\HM$. In the exact same manner as done in \eqref{endofunctor},  one may define the endo-functor of $\yd$:
\[
\ff_A : \yd \rightarrow \yd : Z \mapsto (A \ot Z)^A.
\]
If this functor is an autoequivalence, then $(A\ot \tr)^A$ is an $\tr$-bi-Galois object. At this moment, we don't know whether or not $F_A$ is a tensor autoequivalence of $\yd$ in case $A$ is not $H^*$-Galois over $A_0$.  Davydov and Nikshych constructed a group (iso)morphism Pic$(\C) \ra \Aut^{br}(\mathcal{Z}(\C),\C)$ for any braided tensor category $\C$ (under stronger {conditions}) in \cite{davydovnikshych}. In case $\C = \HM$ they would obtain a group (iso)morphism Pic$(\HM) \ra \Aut^{br}(\yd,\mr)$. We don't know whether their construction of the morphism coincides with ours. If that is the case, we may drop the Galois condition and choose any representative from a Morita equivalence. Nevertheless, one can show that
the functor $F_A$ above is trivializable on $\HM$. Indeed, if $X \in \mr$, \eqref{defazaemod} becomes
\begin{align*}
\gbeg{4}{5}
\got{2}{A\ot\oA}\got{2}{A\ot X} \gnl
\gcl{1}\gbr\gcl{1}\gnl
\gmu\gcl{1}\gcl{1}\gnl
\gwmuh{3}{2}{5}\gcl{1}\gnl
\gvac{1}\gob{1}{A}\gvac{1}\gob{1}{X}
\gend
\end{align*}
or $A \ot X = A \overset{\sim}{\ot} X$, where $(A \overset{\sim}{\ot} - , (-)^A)$ is the equivalence pair of $A$ being Azumaya. Thus $(A \ot X)^A = (A \overset{\sim}{\ot} X)^A \cong X$. 
In other words, the $\Ae$-module structure of $A \ot Z$ in \eqref{defazaemod} is equal to the $\Ae$-module structure of $A \overset{\sim}{\ot} Z$, with the following 'twist' inserted into it,
\begin{diagram}
\gbeg{2}{4}
\got{1}{A}\got{1}{Z}\gnl
\gibrc\gnl
\gbr\gnl
\gob{1}{A}\gob{1}{Z}
\gend
\end{diagram}
which is trivial whenever $Z \in \mr$.
\\\\
Composing the morphism $\ppi : \Br(\HM) \ra \galqc$ with the group isomorphism $\galqc \cong \autbr$ from Proposition \ref{qcbigalisequiv}, we obtain a group morphism 
\[
\Br(\HM) \ra \autbr
\]
which maps a class $[A]$ to the class of natural isomorphisms represented by $F_A$, where $A$ is chosen  as an $H^*$-Galois Azumaya algebra. By construction, the following diagram commutes.
\begin{align*}
\begin{diagram}
1 & \rTo	&	\Br(k)	&\rTo	&	\Br(\HM)			&		\rTo^{\phantom{aaa} \ppi}		&		\gqc{\tr}	\\
&&&&				&		\rdTo^{}		&		\dTo_{\sim}		\\
&&&&				&						&		\autbr
\end{diagram}
\end{align*}
In view of Proposition \ref{picisbr}, we obtain a group morphism
\begin{equation}\label{gpmap}
\tau~: \Pic(\HM) \ra \autbr, \ \ [A]\mapsto F_A, \ \mbox{A is $H^*$-Galois over $A_0$}
\end{equation}
The kernel of this morphism is, by construction, isomorphic to the Brauer group of $k$. The latter is, again by Proposition \ref{picisbr} (replace $H$ by the trivial Hopf algebra $k$), isomorphic to the Picard group Pic$({}_k\M)$ consisting of equivalence classes of one-sided invertible ${}_k\M$-bimodule categories. Let us summarize in the following theorem.

\begin{theorem}\label{bmautbrasequence}
Assume $(H,R)$ is a finite dimensional quasitriangular Hopf algebra. The following diagram commutes:

\begin{align*}
\begin{diagram}[size=2em]
1	& \rTo	&	\Br(k)	&\rTo	&	\BM(k,H,R)			&		\rTo^{\phantom{aa} \ppi}		&		\gqc{\tr}	\\
	&		&	\dTo^{\sim}&		&	\dTo^{\sim}		&										&		\dTo_{\sim}		\\
1	& \rTo	&	\Pic({}_k\M)	&\rTo	&	\Pic(\HM)		&		\rTo^{\tau}							&		\autbr
\end{diagram}
\end{align*}
\end{theorem}

\begin{remark}
The motivation for the construction of the morphism $\tau$ comes from the result \cite[Lemma 4.5]{zhang}. We expect that our construction \eqref{gpmap} of $\tau$ is equal to the construction of the morphism $\Phi: \Pic(\C)\lra \Aut^{br}(\mathcal{Z}(\C),\C)$ by Davydov and Nikshych in \cite[Theorem 4.3]{davydovnikshych} in case $\C=\HM$, or at least the two morphisms are conjugated by group isomorphisms.
If that would be the case, this approach by tensor autoequivalences would render the surjectivity of $\ppi$, i.e. the right exactness of the sequence \eqref{exactss}, at least under the same conditions as in \cite{davydovnikshych}, where the field $k$ is assumed to be algebraically closed of characteristic 0 (so the classical Brauer group $\Br(k)$ is trivial) and the category $\C$ may have to be assumed to be semisimple (or fusion). Another advantage for the morphism $\tau$ would be that  the finiteness requirement for $H$ could be dropped so that the exact sequence \eqref{exactss} works for infinite dimensional QT Hopf algebras. 
\end{remark}

\hide{\begin{align*}
\begin{diagram}
Br(\HM)			&		\rTo^{\phantom{aa} \ppi}		&		\gqc{\tr}	\\	
\dTo^{\sim}		&						&		\dTo_{\sim}		\\
Pic(\HM)		&		\rTo^{}			&		\autbra
\end{diagram}
\end{align*}}

\section*{Acknowledgement}
The authors would like to thank P. Schauenburg for his valuable suggestions and comments, in particular, for pointing out the condition $(\mathbb{A})$ may be dropped in case $\C = \HM$.



\begin{thebibliography}{99}
\bibliographystyle{siam}

\bibitem{andruskiewitschmombelli}
N.~Andruskiewitsch and J.M. Mombelli, \emph{{On module categories over
  finite-dimensional Hopf algebras}}, J. Algebra \textbf{314} (2007), 383--418.

\bibitem{cvozyetter2}
S.~Caenepeel, F.~{Van Oystaeyen}, and Y.H. Zhang, \emph{{The Brauer group of
  Yetter-Drinfel'd module algebras}}, Trans. Amer. Math. Soc. \textbf{349}
  (1997), 3737--3771.

\bibitem{davydovnikshych}
A.~Davydov and D.~Nikshych, \emph{{The Picard crossed module of a braided
  tensor category}}, arXiv: math.QA/1202.0061, 2013.

\bibitem{deligne}
P.~Deligne, \emph{{Cat\'egories tannakiennes}}, The Grothendieck Festschrift
  Volume II (1990), 111--195.

\bibitem{etingofostrik}
P.~Etingof, S.~Gelaki, D.~Nikshych, and V.~Ostrik, \emph{Finite tensor
  categories}, Mosc. Math. J. \textbf{4} (2004), 627--654.

\bibitem{eno}
P.~Etingof, D.~Nikshych, and V.~Ostrik, \emph{{Fusion categories and homotopy
  theory}}, Quantum Topol. \textbf{1} (2010), 209--273.

\bibitem{greenough}
J.~Greenough, \emph{{Monoidal 2-structure of bimodule categories}}, J. Algebra
  \textbf{324} (2010), 1818--1859.
  
\bibitem{Maj} 
S. ~Majid, Foundations of cantum groups, Cambridge University Press 1995.

\bibitem{majidreconstruction}
S.~Majid, \emph{{Reconstruction theorems and rational conformal field
  theories}}, Int. J. Mod Phys. \textbf{6} (1991), 4359--4374.

\bibitem{majidbraidedgroups}
\bysame, \emph{{Braided groups}}, J. Pure Appl. Algebra \textbf{86} (1993),
  187--221.

\bibitem{majidtransmutation}
\bysame, \emph{{Transmutation theory and rank for quantum braided groups}},
  Math. Proc. Camb. Phil. Soc. \textbf{113} (1993), 45--70.

\bibitem{ostrik}
V.~Ostrik, \emph{{Module categories, weak Hopf algebras and modular
  invariants}}, Transform. Groups \textbf{8} (2003), 177--206.

\bibitem{schauenburg2}
P.~Schauenburg, \emph{{Hopf} bigalois extensions}, Comm. Algebra \textbf{24}
  (1996), 3797--3825.

\bibitem{schauenburgbr1}
\bysame, \emph{{Braided bi-Galois theory}}, Ann. Univ. Ferrara Sez. VII
  \textbf{51} (2005), 119--149.

\bibitem{schauenburgbr2}
\bysame, \emph{{Braided bi-Galois theory II: The cocommutative case}}, J.
  Algebra \textbf{324} (2010), 3199--3218.

\bibitem{ulbrich1}
K.-H. Ulbrich, \emph{{Galois extensions as functors of comodules}}, Manuscripta
  Math. \textbf{59} (1987), 391--397.

\bibitem{vanoystaeyenzhangbraided}
F.~{Van Oystaeyen} and Y.H. Zhang, \emph{{The Brauer group of a braided
  monoidal category}}, J. Algebra \textbf{202} (1998), 96--128.

\bibitem{zhang}
{Y.H.} Zhang, \emph{An exact sequence for the {B}rauer group of a finite
  quantum group}, J. Algebra \textbf{272} (2004), 321--378.

\bibitem{ZZ}
Y.H. Zhang and H.X. Zhu, Braided autoequivalences and quantum commutative bi-Galois objects, arXiv:1312.3800.
  
\end{thebibliography}
\end{document}